\documentclass[12pt]{article}         % Dokumentklasse

                          \usepackage[width=16cm, height=22cm]{geometry}   
 \usepackage{amsmath,amssymb}          % Mathesymbole
\usepackage{hyperref} 
 \usepackage[amsmath,thmmarks,hyperref]{ntheorem} 
 \usepackage{icomma}                   
 \usepackage{enumerate}                
 \usepackage{color}                    
 \usepackage{setspace}                  
 \usepackage{needspace}                
 \usepackage{url}                      
 \usepackage{mathrsfs}                 
 \usepackage{graphicx} 
 \usepackage{extpfeil}

 \usepackage{tikz}                     \usetikzlibrary{arrows,decorations.pathmorphing,backgrounds,positioning,fit}
 \usepackage{slashed}
  \usepackage{tikz-cd}                     
\usepackage{verbatim}

\usepackage{todonotes}

\newcounter{comments}
\newenvironment{displaycomment}{\begin{list}{}{\rightmargin=1cm\leftmargin=1cm}\item\sf\begin{small}}{\end{small}\end{list}}

 \numberwithin{equation}{section}
 
 \usepackage[numbers,square]{natbib}
 \usepackage{changepage}

\theoremstyle{nonumberplain}  
\theoremheaderfont{\itshape}  
\theorembodyfont{\normalfont}  
\theoremseparator{.} 
\theoremsymbol{\ensuremath{\Box}}
\newtheorem{proof}{Proof} 
  
\theoremstyle{plain}  
\theoremheaderfont{\normalfont\bfseries}  
\theorembodyfont{\itshape}  
\theoremsymbol{~}
\theoremseparator{.}  
\newtheorem{proposition}{Proposition}[section]  
\newtheorem{corollary}[proposition]{Corollary}  
\newtheorem{lemma}[proposition]{Lemma}  
\newtheorem{theorem}[proposition]{Theorem}     
\newtheorem{question}[proposition]{Question}

\theoremstyle{nonumberplain}  
\newtheorem{maintheorem}[proposition]{Theorem}   

\theoremstyle{plain}  
\theoremheaderfont{\normalfont\bfseries}  
\theorembodyfont{\normalfont}  
 
\newtheorem{remark}[proposition]{Remark}

\newtheorem{example}[proposition]{Example}  
  
\newtheorem{definition}[proposition]{Definition}

\newcommand{\cB}{\mathcal{B}}
\newcommand{\cC}{\mathcal{C}}

\newcommand{\cE}{\mathcal{E}}

\newcommand{\cU}{\mathcal{U}}
\newcommand{\cV}{\mathcal{V}}
\newcommand{\cW}{\mathcal{W}}
\newcommand{\cX}{\mathcal{X}}
\newcommand{\cY}{\mathcal{Y}}
\newcommand{\cZ}{\mathcal{Z}}

\newcommand{\sA}{\mathscr{A}}
\newcommand{\sB}{\mathscr{B}}
\newcommand{\sC}{\mathscr{C}}

\newcommand{\sT}{\mathscr{T}}

\newcommand{\R}{\mathbb{R}}
\newcommand{\Q}{\mathbb{Q}}
\newcommand{\N}{\mathbb{N}}
\newcommand{\Z}{\mathbb{Z}}

\newcommand{\C}{\mathbb{C}}

\DeclareMathOperator{\Aut}{\operatorname{Aut}}
\DeclareMathOperator{\LAut}{\operatorname{LAut}}

\DeclareMathOperator*{\colim}{\mathrm{colim}}

\newcommand{\id}{\mathrm{id}}

\newcommand{\Net}{\textsc{Net}}
\newcommand{\Az}{\textsc{Az}}
\newcommand{\Loc}{\textsc{Loc}}

\newcommand{\Cir}{\mathrm{Cir}}
\newcommand{\BornCoarse}{\textsc{BornCoarse}}
\newcommand{\Cat}{\textsc{Cat}}

\newcommand{\QCA}{\mathrm{QCA}}

\newcommand{\Ad}{\mathrm{Ad}}

\newcommand{\Spectra}{\textsc{Spectra}}

\newcommand{\CMon}{\texttt{CMon}}
\newcommand{\Spaces}{\textsc{Spaces}}

\title{Quantum cellular automata \\ are a coarse homology theory}
\author{ Matthias Ludewig \\ Universit\"at Greifswald}
\begin{document}

\maketitle

\begin{abstract}
We show that quantum cellular automata naturally form the degree-zero part of a coarse homology theory. 
The recent result of Ji and Yang that the space of QCA forms an $\Omega$-spectrum in the sense of algebraic topology is a direct consequence of the formal properties of coarse homology theories.
\end{abstract}

\section{Introduction}

Quantum cellular automata (QCA) are certain locality preserving automorphisms of $C^*$-algebras that arise as tensor products over matrix algebras attached to points in space.
In a very recent paper \cite{QCAspace}, Ji--Yang prove strong structural results on the  space of QCA, which they call the \emph{QCA conjecture}: The space of QCA on $\Z^n$ is homotopy equivalent to the loop space of the space of QCA on $\Z^{n+1}$.
In this paper, we give a somewhat simplified proof, as well as a conceptual reason for this equivalence: QCA form a coarse homology theory.

\medskip

To define QCA on non-compact spaces, one typically starts with a discrete metric space $X$ and attaches matrix algebras $\sA_x$ to each point in space.
Then, an automorphism $\alpha$ of the infinite tensor product
\[
\sA_X := \bigotimes_{x \in X} \sA_x
\]
is a QCA if there exists $R>0$ such that $\alpha$ sends each of the subalgebras $\sA_x$ to the tensor product of $\sA_y$ over those $y$ in an $R$-neighborhood around $x$.
One of the main points I would like to make in this paper is that the ``correct'' type of space to consider for studying QCA is \emph{not} a metric space but a \emph{bornological coarse space}:
The observation is that only the large scale structure of space matters for QCAs;
however, next to the large scale (coarse) structure, metric spaces also carry an unnecessary small scale structure (a topology, even a uniform structure).

\medskip

A bornological coarse space is a set $X$ together with the information on which sets are to be considered \emph{bounded}, as well as a collection of subsets of $X \times X$ which are called \emph{entourages} or \emph{controlled subsets} and which encode uniform bounds from above on distances (see \S\ref{PreliminariesCoarse} for precise definitions).
The collection $\cB$ of bounded subsets of a bornological coarse space $X$ is partially ordered by inclusion, and a \emph{net} on $X$ is just a functor from the poset $\cB$ to the category of finite-dimensional $C^*$-algebras and injective $*$-homomorphisms, which satisfies a certain cocontinuity hypothesis.
Particularly nice examples are the \emph{local matrix nets} or \emph{spin systems}, which are given by a tensor product of matrix algebras indexed by a locally finite subset of sites.
The crucial observation is now that the notion of QCA can be formulated using only the coarse structure: An automorphism $\alpha$ of a local matrix net is a QCA if there exists an entourage $E \subset X \times X$ such that $\alpha(\sA_x)$ commutes with $\sA_y$ unless $(x, y) \in E$; in other words, only pairs of points from $E$ interact via $\alpha$.

\medskip

The reason for working with the larger class of coarse spaces instead of the more rigid framework of metric spaces is not the desire to study QCA on strange coarse spaces.
Instead, the reason is that---in contrast to the category of metric spaces---the category of bornological coarse spaces is well-behaved enough to do homotopy theory with it \cite{BunkeEngel2020}. 
This leads to the notion of a \emph{coarse homology theory}, which is a functor from bornological coarse spaces to spectra that sends close maps to homotopic maps and satisfies a Mayer--Vietoris axiom for decompositions of spaces.
Important examples include coarse $K$-homology and coarse ordinary homology.
Using the Mayer--Vietoris property of coarse homology theories, the QCA hypothesis $Q(X \otimes \Z) \cong \Omega Q(X)$ from \cite{QCAspace} is a direct consequence of the following theorem, which is the main result of this paper and identifies the QCA group as a generalized homo\-logy invariant of large-scale geometry.

\begin{maintheorem}
There is a coarse homology theory $Q$ such that the degree zero group $Q_0(X)$ is the group $\QCA(X)$ of quantum cellular automata on $X$.
\end{maintheorem}

The category of local matrix nets is too rigid for the homotopical constructions required for the construction of $Q$.
To remedy this, inspired by \cite{QCAspace}, we introduce in this paper the new notion of an \emph{Azumaya net}, which plays a role analogous to Azumaya algebras in algebra.
These are nets $\sA$ such that there exists another net $\sA'$ on the same space such that we have a (controlled) isomorphism
\[
\sA \otimes \sA' \cong \sB
\]
of nets, where $\sB$ is a local matrix net.
This is a straightforward generalization of the notion of an Azumaya algebra over a commutative ring $R$, which is an $R$-algebra $A$ such that there exists an $R$-algebra $A'$ with $A \otimes_R A' \cong M_n(R)$.
Denoting by $\Az(X)$ the symmetric monoidal category of Azumaya nets over $X$, there are natural maps $K(\Az(X)) \to \Omega K(\Az(X \otimes \R))$, where $K$ denotes the algebraic $K$-theory spectrum for symmetric monoidal categories.
The spectrum $Q(X)$ is then defined as
\[
Q(X) = \colim_{n \to \infty} \Omega^{n+1} K(\Az(X \otimes \R^n)),
\]
where the colimit is taken with respect to the maps above.
One further new result that follows directly from this construction is that we have a canonical isomorphism
\[
\QCA(\Z^n) \cong K_0(\Az(\Z^{n-1})).
\]
In other words, the group of QCA on $\Z^n$ is identified with the Grothendieck group of Azumaya nets in one dimension lower.
For $n=1$, this isomorphism is precisely the GNVW index \cite{GrossNesmeVogtsWerner2012}, but in higher dimensions, this seems to be new.

\paragraph{Acknowledgements.}

I would like to thank Christoph Winges, Daniel Kasprowski and Ulrich Bunke for helpful discussions.
I also gratefully
acknowledge support from SFB 1085 ``Higher invariants'' funded by the German Research
Foundation (DFG).

\tableofcontents

\section{Nets and QCA on coarse spaces}

In this section, we give a brief recollection of the basic notions of coarse geometry, before developing the theory of nets on bornological coarse spaces and giving the definition of the group of quantum cellular automata.

\subsection{Preliminaries on coarse spaces}
\label{PreliminariesCoarse}

\begin{definition}
Let $X$ be a set.
\begin{enumerate}[(1)]
\item
A \emph{bornology} on a set $X$ is a collection $\cB$ of subsets of $X$ (called \emph{bounded} subsets) which is closed under taking finite unions and subsets and with the property that $\bigcup \cB = X$.
A \emph{bornological space} is a set together with a bornology.
\item
A \emph{coarse structure} on $X$ is a collection $\cC$ of subsets of $X \times X$ (called \emph{entourages}), which contains the diagonal
\[
\Delta = \{(x, x) \mid x \in X\}
\] 
and which is closed under taking subsets, finite unions, compositions
\[
E \circ F = \{(x, z) \in X \times X \mid \exists y \in X : (x, y) \in F, (y, z) \in E\}
\]
and inverses 
\[
E^{-1} = \{(y, x) \in X \times X \mid (x, y) \in E\}.
\]
A \emph{coarse space} is a set together with a coarse structure.
\end{enumerate}
\end{definition}

For a subset $Y \subseteq X$ of a coarse space $X$ and an entourage $E$, we denote the \emph{$E$-fattening of $Y$} by
\[
Y_E = \{ x \in X \mid \exists y \in Y : (x, y) \in E\}.
\]	

\begin{definition}[bornological coarse spaces \cite{BunkeEngel2020}]
A \emph{bornological coarse space} is a set $X$ together with a bornology $\cB$ and a coarse structure $\cC$, which are compatible in the sense that for any $B \subseteq X$ and any entourage $E \in \cC$, the fattening $B_E$ is bounded again.
\end{definition}

We usually just write $X$ instead of $(X, \cB, \cC)$ for a bornological coarse space, suppressing the bornology and coarse structure in notation.

\begin{remark}
If $\cC$ is a coarse structure on a set $X$, then there is a canonical bornology determined by the coarse structure, consisting of those subsets $B \subseteq X$ such that $B \times B$ is an entourage of $\cC$.
Hence any coarse space is canonically a bornological coarse space.
However, bornological coarse spaces allow for more flexibility with respect to the bornology, which is sometimes useful.
\end{remark}

\begin{example}
If $\cE$ is any collection of subsets of $X \times X$  with $\Delta \subseteq \bigcup\cE$, then one can consider the smallest coarse structure $\cC = \langle \cE\rangle$ containing $\cE$, the \emph{coarse structure generated by $\cE$}.
\end{example}

\begin{example}
If $d$ is a (pseudo-) metric on $X$, then the \emph{metric coarse structure} is the coarse structure generated by the entourages $\{(x, y) \in X \times X \mid d(x, y) \leq R\}$ for $R \geq 0$.
The bornology induced by this coarse structure is the collection of subsets of $X$ with finite $d$-diameter.
Whenever we view metric spaces as coarse spaces (such as $\Z^n$ or $\R^n$), we always understand them to be equipped with the metric coarse structure and the induced bornology, unless otherwise specified.
\end{example}

\begin{example}
\label{TensorProductBornCoarse}
Given two bornological coarse spaces $X$ and $X'$, there is a bornological coarse space $X \otimes X'$ which has underlying set $X \times X'$, bornology $\cB \times \cB'$ and coarse structure $\cC \times \cC'$, the elementwise cartesian product. 
As noted in \cite[Example 2.32.]{BunkeEngel2020}, the bornological coarse space $X \otimes X'$ is \emph{not} the categorical product of $X$ and $X'$; the latter does exist and has the same underlying set and coarse structure, but a different bornology.
The product $\otimes$ yields an auxiliary symmetric monoidal structure on $\BornCoarse$.
In particular, we have canonical identifications
\[
\R^{n-1} \otimes \R = \R^n, \qquad \Z^{n-1} \otimes \Z \cong \Z^n.
\]
\end{example}

\begin{definition}[coarse maps]
Let $X$ and $X'$ be bornological coarse spaces and let $f : X \to X'$ be a map.
\begin{enumerate}[(1)]
	\item $f$ is called \emph{proper} if for every $B' \subseteq X'$ bounded, $f^{-1}(B') \subseteq X$ is bounded.
	\item $f$ is \emph{controlled} if for every entourage $E$ of $X$, $(f \times f)(E)$ is an entourage of $X'$.
	\item $f$ is called \emph{coarse} if it is proper and controlled.
\end{enumerate}
We denote by $\BornCoarse$ the category of bornological coarse spaces and coarse maps. 
\end{definition}

What allows us to do homotopy theory with bornological coarse spaces is the following notion of equivalence for coarse maps.

\begin{definition}[closeness]
Let $X$ and $X'$ be coarse spaces.
\begin{enumerate}[(1)]
	\item $f, g : X \to X'$ are \emph{close} if the set $(f \times g)(\Delta) = \{(f(x), g(x)) \mid x \in X\}$ is an entourage of $X'$.
	\item A coarse map $f : X \to X'$ is a \emph{coarse equivalence} if there exists a coarse map $g : X' \to X$ such that the compositions $f \circ g$ and $g \circ f$ are each close to the identity.
\end{enumerate} 
\end{definition}

\begin{example}
For any dimension $n$, the inclusion maps $\Z^n \hookrightarrow \R^n$ are coarse equivalences.
However, $\R^n$ is not coarsely equivalent to $\R^m$ for $n \neq m$.
\end{example}

To formulate support conditions in coarse geometry that are invariant under coarse fattenings, we use the notion of a big family, defined as follows.

\begin{definition}
	A \emph{big family on} $X$ is a collection $\cY$ of subsets of $X$ that is closed under taking subsets, finite unions and fattenings.
\end{definition}

Any subset $Y \subseteq X$ generates a big family of $X$ that we denote by $\{Y\}$, given by
\[
\{Y\} = \{Z \subseteq X \mid \exists E \in \cC : Z \subseteq Y_E\}.
\]
If $\cY$ and $\cZ$ are big families, then their elementwise union and intersection
\begin{align*}
\cY \Cup \cZ &= \{ Y \cup Z \mid Y \in \cY, Z \in \cZ\}
\\
\cY \Cap \cZ &= \{ Y \cap Z \mid Y \in \cY, Z \in \cZ\}
\end{align*}
are again big families.

\medskip

In the ordinary homotopy theory of topological spaces, contractible spaces are \emph{motivically trivial}, in the sense that they evaluate trivially on any homology theory. 
In the theory of coarse spaces, this role is played by so-called flasque spaces:

\begin{definition}[flasqueness]
A bornological coarse space $X$ is called \emph{flasque} if there exists a coarse map $f : X \to X$ with the following properties:
\begin{enumerate}[(1)]
\item $f$ is close to the identity map of $X$;
	\item For every bounded set $B \subseteq X$, there exists $n \in \N$ such that $f^n(X) \cap B = \emptyset$, where $f^n$ is the $n$-th iterate of $f$;
	\item 
	For every entourage $E$, the set
	\[
	\bigcup_{n \in \N} (f^n \times f^n)(E)
	\]
	is an entourage.
\end{enumerate}
Such a map $f$ is called a \emph{map witnessing the flasqueness}.
\end{definition}

\subsection{Nets on coarse spaces}

Let $X$ be a bornological coarse space.
The bornology $\cB$ of $X$ is a partially ordered set and hence may be viewed as a category with exactly one morphism for each inclusion $B \subseteq B'$.

\begin{definition}[nets]
A \emph{net} on $X$ is a functor
\[
\sA : \cB \longrightarrow \left\{\substack{
\text{category of}
\\
\text{finite-dimensional}
\\
\text{$C^*$-algebras}
\\
\text{and injective}
\\
\text{ $*$-homomorphisms}} \right\}, \quad~~ B \longmapsto \sA_B.
\]
with the following property:
Whenever $B \subseteq X$ is bounded and $\cB' \subseteq \cB$ is a directed subset such that $\bigcup \cB' = B$, then canonical map
\begin{equation}
\label{PropertyNet}
\colim_{B' \in \cB'} \sA_{B'} \longrightarrow \sA_B
\end{equation}
is an isomorphism.
\end{definition}

The property \eqref{PropertyNet} of the functor $\sA$ is there to avoid pathological phenomena, see Remark~\ref{RemarkLocallyFiniteSupport} below.
Given a net $\sA$, we define for any subset $Y \subseteq X$
\[
\sA_Y := \colim_{\substack{B \in \cB \\B \subseteq Y }} \sA_B,
\]
where the colimit is taken in the category of $*$-algebras.
We could also take the limit in the category of $C^*$-algebras (which would amount to taking the completion of $\sA_Y$), but for the current paper it is more convenient to work with uncompleted algebras.
By \eqref{PropertyNet}, we can as well take the colimit over an arbitrary directed system of bounded subsets covering $B$ instead.

Since the structure maps of $\sA$ are required to be injective, all the algebras $\sA_Y$ for $Y \subseteq  X$ may be viewed as subalgebras of $\sA_X$, and making these identifications, $\sA_X$ is just the union of all $\sA_B$.

\begin{definition}
If $\sA$ and $\sB$ are nets on $X$, their \emph{tensor product} $\sA \otimes \sB$ is defined by
\[
(\sA \otimes \sB)_B := \sA_B \otimes \sB_B.
\]
\end{definition}

\begin{definition}[homomorphisms]
Let $\sA$, $\sB$ be two nets.
A \emph{homomorphism} of nets $\varphi : \sA \to \sB$ is a $*$-homomorphism $\varphi: \sA_X \to \sB_X$ which is controlled in the sense that there exists an entourage $E$ of $X$ such that for all bounded sets $B \subseteq X$, we have
\[
\varphi(\sA_B) \subseteq \sB_{B_E}.
\]
In this case, we say more precisely, that $\varphi$ is \emph{$E$-controlled}.
We call $\varphi$ \emph{local} if it is $\Delta$-controlled.
\end{definition}

Let $\varphi : \sA \to \sB$ be a homomorphism of nets such that its controlled $*$-homomorphism $\sA_X \to \sB_X$ is an isomorphism. 
If the inverse $*$-homomorphism is again controlled, then we say that $\varphi$ is an \emph{isomorphism} of nets.
We denote by $\Net(X)$ the groupoid of nets on $X$ whose morphisms are (controlled) isomorphisms of nets.

\begin{definition}[types of nets]
Let $\sA$ be a net on $X$.
\begin{enumerate}[(1)]
\item $\sA$ is \emph{semilocal} if whenever $Y, Y' \subseteq X$ satisfy $Y \cap Y' = \emptyset$, then $\sA_Y$ and $\sA_{Y'}$ commute inside $\sA_{Y \sqcup Y'}$.
	\item 
	$\sA$ is \emph{local} if it is semilocal and whenever $Y, Y' \subseteq X$ satisfy $Y \cap Y' = \emptyset$, then the inclusion maps induce an isomorphism 
\[
\sA_Y \otimes \sA_{Y'} \cong \sA_{Y \sqcup Y'}.
\]
	\item 
	$\sA$ is called a \emph{matrix net} if $\sA_B$ is a full matrix algebra for each $B \subseteq X$ bounded.
	\item 
	$\sA$ is called \emph{Azumaya} if there exists a net $\sA'$ on $X$ and a local matrix net $\sB$ on $X$ such that 
	\[
	\sA \otimes \sA' \cong \sB.
	\]
\end{enumerate}
We write
\[
\Loc(X) \subseteq \Az(X) \subseteq \Net(X) 
\]
for the full subgroupoids of local matrix nets, respectively Azumaya nets.
\end{definition}

\begin{example}
\label{ExampleCoordinateNets}
Whenever $q : X \to \N$ is a function with locally finite support (meaning that for each bounded set $B \subseteq X$, we have $q(x) = 1$ for all but finitely many $x \in B$), we obtain a local matrix net by setting
\[
\sA_B := \bigotimes_{x \in B} M_{q(x)}(\C),
\]
where $M_q(\C)$ denotes the algebra of complex $q \times q$ matrices.
This is the class of nets considered in \cite{QCAspace}.
Any local matrix net is isomorphic in $\Net(X)$ to one of this form.
The more general class of local matrix nets may be viewed as a ``coordinate free version'' of these nets.
\end{example}

\begin{remark}
While the terminology ``local matrix net'' seems more appropriate in our context, we remark that these are often called \emph{spin system} in the mathematical physics literature.
\end{remark}

\begin{remark}
It follows from Corollary~\ref{CorollaryAzumayaSubnet} that any Azumaya net on $X$ is isomorphic to a semilocal net $\sA$, because subnets of local nets are necessarily semilocal.
Throughout, we may therefore replace $\Az(X)$ by the full subcategory of semilocal Azumaya nets.
\end{remark}

\begin{definition}[support]
Let $\sA$ be a net on $X$.
A subset $Y \subseteq X$ is a \emph{support} for $\sA$ if 
for every $B \in \cB$, we have $\sA_B = \sA_{Y \cap B}$.
If $\cY$ is a big family in $X$, we write
\[
\Loc(\cY) \subseteq \Loc(X), \qquad  \Az(\cY) \subseteq \Az(X), \qquad \Net(\cY) \subseteq \Net(X)
\]
for the full subcategories of those nets that are supported on some member $Y$ of $\cY$.
\end{definition}

\begin{remark}
\label{RemarkLocallyFiniteSupport}
Any local net has a locally finite support $Y$, meaning that $Y \cap B$ is finite for each $B \subseteq X$ bounded.
Indeed, by locality, we have 
\[
\sA_F = \bigotimes_{x \in F} \sA_x
\]
for each finite subset $F \subseteq X$.
Now by \eqref{PropertyNet}, we have
\[
\sA_B = \colim_{F \subseteq B ~\text{finite}} \sA_F = \bigotimes_{x \in B} \sA_x.
\]
Since $\sA_B$ is finite-dimensional, we have $\sA_x \neq \C$ for only finitely many $x \in B$.
Without \eqref{PropertyNet} on the functor $\sA$, there are pathological examples of local nets without locally finite support using the axiom of choice.
Namely, if $\cU$ is a free ultrafilter on $X$ (which always exists as soon as $X$ is infinite), we may define a function $\mu : \cB \to \{0, 1\}$ by $\mu(B) = 0$ if $B \notin \cU$ and $\mu(B) = 1$ if $B \in \cU$.
This function is additive because filters cannot contain two disjoint sets.
Therefore, setting
\[
\sA_B := M_2(\C)^{\otimes \mu(B)}
\]
defines a local net on $X$.
$\sA$ does not have locally finite support because $\cU$ does not contain any finite sets.
\end{remark}

If $\sA$ is a net on $Y\subseteq X$, we obtain a net $\tilde{\sA}$ on $X$ given by $\tilde{\sA}_B := \sA_{B \cap Y}$.
This gives identifications of nets on $Y$ with nets on $X$.
Using these identifications, we have canonical isomorphisms
\begin{equation}
\label{NetsOnFamiliesAsColimits}
\begin{aligned}
 \Az(\cY) &\cong \colim_{Y \in \cY} \Az(Y), \\ 
 \Loc(\cY) &\cong \colim_{Y \in \cY} \Loc(Y), \\
 \Net(\cY) &\cong \colim_{Y \in \cY} \Net(Y).
 \end{aligned}
\end{equation}

\subsection{Nets and coarse maps}

If $f : X \to X'$ is a coarse map, then we may form the pushforward net $f_*\sA$, given by $(f_*\sA)_{B'} = \sA_{f^{-1}(B')}$ for $B' \subseteq X'$ bounded, which is a net on $X'$.
This is well defined because $f$ is proper.

\medskip

We have
\[
(f_*\sA)_{X'} 
= \colim_{B' \in \cB'} (f_*\sA)_{B'}
= \colim_{B' \in \cB'} \sA_{f^{-1}(B')} = \sA_X,
\]
since $\{f^{-1}(B') \mid B' \in \cB'\} \subseteq \cB$ is cofinal.
Hence if $\varphi : \sA \to \sB$ is a morphism of nets, then we obtain a $*$-homomorphism 
\[
f_* \varphi : (f_*\sA)_X = \sA_X \xrightarrow{~~\varphi_X~~} \sB_X = (f_*\sB)_X.
\]
If $\varphi$ is $E$-controlled, then for each $B' \subseteq X'$ bounded, we have
\[
\varphi((f_*\sA)_{B'}) = \varphi(\sA_{f^{-1}(B')}) \subseteq \sB_{f^{-1}(B')_E} \subseteq \sB_{f^{-1}(B'_{(f\times f)(E)})} = (f_*\sB)_{B'_{(f\times f)(E)}},
\]
hence $f_* \varphi$ is $(f\times f)(E)$-controlled.

\medskip

The pushforward functor $f_*$ respects tensor products, in the sense that there are canonical isomorphisms
\[
f_*(\sA \otimes \sB) \cong f_*\sA \otimes f_* \sB.
\]
%
\begin{comment}
\begin{lemma}
Let $f : X \to X'$ be a coarse map.
If the net $\sA$ on $X$ is coarsely local, then $f_*\sA$ is coarsely local as well.
Similarly, if $\sA$ is local, then $f_*\sA$ is local as well.
\end{lemma}

\begin{proof}
Let $Y' \subseteq X'$ be arbitrary.
Let $E$ be an entourage of $X$ controlling the coarse locality of $\sA$.
Then there exists a tensor product decomposition $\sA \cong \sB \otimes \sB'$ with $\sB$ supported on $f^{-1}(Y')_E$ and $\sB'$ supported on $(X \setminus f^{-1}(Y'))_E$.
We then get
\[
f_*\sA \cong f_*\sB \otimes f_*\sB'.
\]
Since
\begin{align*}
f\bigl(f^{-1}(Y')_E\bigr) 
&= \bigl\{ f(x) \mid \exists y \in X :  (x, y) \in E, f(y) \in Y'\bigr\} = Y'_{(f\times f)(E)}
\end{align*}
the factor $f_*\sB$ is supported on $Y'_{(f\times f)(E)}$ and a similar calculation shows that $f_*\sB'$ is supported on $(X' \setminus Y')_{(f\times f)(E)}$.
That $f_*\sA$ is local if $\sA$ is local follows from setting $E = \Delta$ in the above calculation.
\end{proof}
\end{comment}
%
Combining the above observations, we obtain that for each coarse map $f: X \to X'$, we get a symmetric monoidal functor
\begin{align*}
f_* : \Net(X) &\longrightarrow \Net(X').
\end{align*}
One easily checks that the pushforward of a local net is again local. It follows that $f_*$ also sends Azumaya nets to Azumaya nets, so $f_*$ restricts to functors
\begin{align*}
f_* :\Az(X) &\longrightarrow\Az(X'),
 \\
f_* :\Loc(X) &\longrightarrow\Loc(X').
\end{align*}
If $f: X \to X'$ and $g : X' \to X''$ are two coarse maps, then we have obvious natural isomorphisms of functors
\[
(g \circ f)_* \cong g_*  f_*.
\]
These structures assemble to (pseudo-)functors
\[
\Net, \Az, \Loc : \BornCoarse \longrightarrow \CMon(\Cat)
\]
from bornological coarse spaces to the bicategory of symmetric monoidal categories (i.e., commutative monoid objects in $\Cat$).

\begin{proposition}
\label{PropCloseness}
Suppose that $f, g : X \to X'$ are close. 
Then we have natural isomorphisms of functors
\[
f_* \cong g_* : \Net(X) \to \Net(X').
\]
Since they are full subcategories, the same is true for the pushforward functors on Azumaya and local nets.
\end{proposition}

\begin{proof}
For any net $\sA$ on $X$, we have the identification  $(f_*\sA)_X = \sA_X = (g_*\sA)_X$ of the global algebras. 
This is clearly natural in $\sA$, but we need to show that this isomorphism is controlled.
To this end, let $E = (g \times f)(\Delta)$, which is an entourage of $X'$ as $f$ and $g$ are close.
Then 
\begin{align*}
g^{-1}(B'_E) 
&= \{z \in X \mid g(z) \in B'_E\}
\\
&= \{z \in X \mid \exists x' \in B' : (g(z), x') \in E\}
\\
&= \{x \in X \mid f(x) \in B'\}
\\
&= f^{-1}(B'),
\end{align*}
hence
\[
(f_*\sA)_{B'} = \sA_{f^{-1}(B')} = \sA_{g^{-1}(B'_E)} = (g_*\sA)_{B'_E}.
\]
Hence the canonical $*$-isomorphism $(f_*\sA)_X \to (g_*\sA)_X$ is controlled by $E$.
Swapping the roles of $f$ and $g$, we see that the inverse is controlled by the inverse entourage $E^{-1}$.
\end{proof}

\begin{corollary}
\label{CorollaryEquivalenceBigFamily}
For any subset $Y \subseteq X$, the canonical functor
\[
\Net(Y)\longrightarrow \Net(\{Y\})
\]
is an equivalence.
The same is true for the subcategories of Azumaya, respectively local nets.
\end{corollary}

\begin{proof}
This follows from Prop.~\ref{PropCloseness}, as for all entourages $E$ and $F$ of $X$ with $E \subseteq F$, the inclusion map $Y_E \hookrightarrow Y_F$ is a coarse equivalence.
\end{proof}

\begin{proposition}
\label{PropFlasque}
If $X$ is flasque, then there exists a symmetric monoidal and faithful endofunctor $S : \Net(X) \to \Net(X)$ together with a natural isomorphism
\[
S \cong \id \otimes S.
\]
This functor restricts to a functor on $\Az(X)$ and $\Loc(X)$ (and so does the natural transformation because both are full subcategories).
\end{proposition}

\begin{proof}
Let $f: X \to X$ be the map witnessing flasqueness. 
For any net $\sA$ on $X$, we set
\begin{equation*}
S(\sA) := \sA \otimes f_*\sA \otimes f_*^2 \sA \otimes f_*^3 \sA \otimes \cdots.
\end{equation*}
This infinite tensor product is to be interpreted as
\begin{equation*}
S(\sA)_B = \bigotimes_{\substack{ n = 0 \\ f^{-n}(B) \cap B \neq \emptyset}}^\infty \sA_{f^{-n}(B)},
\end{equation*}
which is actually a finite tensor product as $f^n(X) \cap B = \emptyset$ for all but finitely many $n$.
This defines a net on $X$.
If $\sA$ is Azumaya or local matrix, then $S(\sA)$ has these properties as well.

\begin{comment}
To see that $S(\sA)$ is coarsely local if $\sA$ is, let $E$ be a locality preserving entourage for the net $\sA$ and let $Y \subseteq X$ be arbitrary.
Then for every $n \in \N$, there exists a tensor decomposition
\[
\sA = \sB^{(n)} \otimes {\sB'}^{(n)}
\]
with $\sB^{(n)}$ supported on $f^{-n}(Y)_E$ and ${\sB'}^{(n)}$ supported on $f^{-n}(X\setminus Y)_E$.
Hence $f^n_* \sB^{(n)}$ is supported on $Y_{(f^n \times f^n)(E)}$ and $f^n_* {\sB'}^{(n)}$ is supported on $(X \setminus Y)_{(f^n \times f^n)(E)}$.
In total, we get a tensor decomposition $S(\sA) = \sB \otimes \sB'$ with
\[
\sB := \sB^{(0)} \otimes f_* \sB^{(1)} \otimes f^2_*\sB^{(2)} \otimes \cdots, \quad~ \text{and} \quad~  \sB' := {\sB'}^{(0)} \otimes f_* {\sB'}^{(1)} \otimes f^2_*{\sB'}^{(2)} \otimes \cdots,
\]
supported on $Y_F$, respectively $(X \setminus Y)_F$. 
Here 
\begin{equation*}
F = \bigcup_{n=0}^\infty (f^n\times f^n)(E),
\end{equation*}
which is an entourage by properties of the map $f$.
\end{comment}

If $\varphi : \sA \to \sB$ is an $E$-controlled morphism, then we obtain a $*$-homomorphism $S(\sA)_X \to S(\sB)_X$ by applying $\varphi$ to each tensor factor.
This homomorphism is controlled by
\begin{equation*}
\bigcup_{n=0}^\infty (f^n\times f^n)(E),
\end{equation*}
which is an entourage by the properties of $f$, hence it  defines a homomorphism of nets.

Finally, since $f$ is close to the identity, we have a natural isomorphism $S(\sA) \cong f_*S(\sA)$ by Prop.~\ref{PropCloseness}, which in turn is clearly naturally isomorphic to $\sA \otimes S(\sA)$, in fact by a local homomorphism.
\end{proof}

\subsection{Tensor factors and homomorphisms}

Let $\sB$ be a net on $X$.
A \emph{subnet} of $\sB$ is a net $\sA$ on $X$ such that $\sA_B \subseteq \sB_B$ for each $B \subseteq X$ bounded.
We write $\sA \subseteq \sB$ if $\sA$ is a subnet of $\sB$.

\begin{definition}
Let $\sB$ be a net and let $\sA \subseteq \sB$ be a subnet.
The \emph{commutant} of $\sA$ in $\sB$ is the net $\sA'$ given by
\[
\sA'_B = (\sA_X)' \cap \sB_B,
\]
where $(\sA_X)' = \{b \in \sB_X \mid \forall a \in \sA_X : ab = ba\}$ is the commutant of $\sA_X$ in $\sB_X$.
$\sA$ is a \emph{tensor factor} of $\sB$ if the canonical homomorphism $\sA \otimes \sA' \to \sB$ is an isomorphism of nets. 
\end{definition}

\begin{comment}
\begin{example}
%\label{ExampleLocalEmbeddings}
If $\sA$ and $\sB$ are matrix nets and $\varphi : \sA \hookrightarrow \sB$ is a \emph{local} homomorphism of nets, then $\sA$ is a tensor factor of $\sB$, with the commutant $\sA'$ again being a matrix net. 
This follows from the general fact that if $B$ is a matrix algebra and $A \subseteq B$ is subalgebra containing the unit that is also isomorphic to a matrix algebra, then we have $B = A \otimes A'$, with the commutant $A'$ also being a matrix algebra.
\end{example}
\end{comment}

Notice that $\sA'_X = (\sA_X)'$ and that multiplication induces a canonical algebra homomorphism $\mu : \sA_X \otimes (\sA_X)' \to \sB_X$ which restricts to $*$-homomorphisms
\begin{equation}
\label{LocalStarHom}
(\sA \otimes \sA')_B =\sA_B \otimes (\sA_X)' \cap \sB_B \longrightarrow \sB_B
\end{equation}
for every $B \subseteq X$.
However, it is generally not clear that these maps are isomorphisms, even if $\mu$ is an isomorphism globally.
We have the following criterion to have an isomorphism of nets, compare \cite[Lemma 63]{QCAspace}.

\begin{lemma}
\label{LemmaCharacterizationTensorFactor}
A subnet $\sA \subseteq \sB$ of a net $\sB$ is a tensor factor if and only if the multiplication map  $\mu : \sA_X \otimes (\sA_X)' \to \sB_X$ is injective and there exists an entourage $E$ such that
\begin{equation}
\label{ControlledInclusion}
\sB_B \subseteq \sA_{B_E} \cdot  \sA'_{B_E}.
\end{equation}
for all bounded sets $B \subseteq X$.
\end{lemma}

\begin{proof}
($\Longrightarrow$) Suppose that $\sA$ is a tensor factor of $\sB$.
Then multiplication yields a controlled $*$-homomorphism $\mu : (\sA \otimes \sA')_X \cong \sA_X \otimes (\sA_X)' \to \sB_X$ whose inverse is controlled as well. 
If $E$ is a control for the inverse, then we have  
\[
\sB_B \subseteq \mu\bigl((\sA \otimes \sA')_{B_E}\bigr) = \sA_{B_E} \cdot \sA_{B_E}'.
\]
($\Longleftarrow$) 
Suppose that $\mu$ is injective and that \eqref{ControlledInclusion} holds.
By \eqref{ControlledInclusion}, we have
\[
\sB_X = \bigcup_{B \in \cB} \sB_B \subseteq \bigcup_{B \in \cB} \sA_{B_E} \cdot \sA_{B_E}' = \sA_X \otimes (\sA_X)'.
\]
This shows that $\mu$ is also surjective, hence an isomorphism.
$\mu$ is always controlled (in fact, local) and \eqref{ControlledInclusion} states precisely that the inverse is $E$-controlled.
\end{proof}

\begin{definition}
If $\sA, \sB$ are nets and $\varphi : \sA \to \sB$ is a homomorphism, then the \emph{image net} $\varphi_*\sA$ is defined by
\[
(\varphi_*\sA)_B = \varphi(\sA_X) \cap \sB_B.
\]
\end{definition}

\begin{remark}
\label{RemarkAisomorphicToImage}
The functor $\varphi_*\sA$ is again a net because the intersection of two $C^*$-subalgebras is again a $C^*$-subalgebra.
If $\varphi$ is $E$-controlled, then $\varphi(\sA_B) \subseteq \sB_{B_E}$, hence also $\varphi(\sA_B) \subseteq (\varphi_*\sA)_{B_E}$ and $(\varphi_*\sA)_B \subseteq \varphi(\sA_{B_E})$.
We therefore obtain an $E$-controlled isomorphism of nets $\sA \cong \varphi_*\sA$.
\end{remark}

The following structure results on tensor factors will be used in \S\ref{SectionMayerVietoris} to prove the Mayer--Vietoris decomposition result.

\begin{proposition}
[Images of tensor factors]
\label{PropImageOfTensorFactor}
Let $\sB$ and $\tilde{\sB}$ be nets and let $\varphi: \sB \to \tilde{\sB}$ be an isomorphism of nets.
If $\sA$ is a tensor factor of $\sB$ with commutant $\sA'$, then $\varphi_*\sA$ is a tensor factor of $\tilde{\sB}$ with commutant $\varphi_*\sA'$.
\end{proposition}

\begin{proof}
Because $\varphi$ is an isomorphism, we have $\varphi((\sA_X)') = \varphi(\sA_X)'$.
Hence for bounded subsets $B \subseteq X$, we get
\[
(\varphi_*\sA')_B = \varphi((\sA_X)') \cap \tilde{\sB}_B = \varphi(\sA_X)' \cap \tilde{\sB}_B = \bigl((\varphi_*\sA)_X\bigr)' \cap \tilde{\sB}_B = (\varphi_*\sA)'_B,
\]
so $\varphi_*\sA'$ is indeed the commutant of $\varphi_*\sA$.
We may therefore verify the criterion of Lemma~\ref{LemmaCharacterizationTensorFactor}.

Since $\sA$ is a tensor factor of $\sB$, the multiplication map $\sA_X \otimes (\sA_X)' \to \sB_X$ is an isomorphism.
Since $\varphi$ is an isomorphism, we obtain that also the multiplication map
\[
(\varphi_*\sA)_X \otimes (\varphi_*\sA)'_X = \varphi(\sA_X) \otimes \varphi(\sA'_X)\longrightarrow \varphi(\sB_X) = \tilde{\sB}_X
\]
is an isomorphism.
Moreover, by  Lemma~\ref{LemmaCharacterizationTensorFactor}, there exists an entourage $E$ such that
\[
\sB_B \subseteq \sA_{B_E} \cdot \sA'_{B_E}
\] 
for all $B \subseteq X$ bounded.
If $F$ is a control for both $\varphi$ and $\varphi^{-1}$, then 
%$\varphi(\sA_B) \subseteq \tilde{\sB}_{B_F}$ and $\varphi^{-1}(\tilde{\sB}_B) \subseteq \sB_{B_F}$, hence
\[
\varphi(\sA_B) \subseteq \varphi(\sA_X) \cap \tilde{\sB}_{B_F} = (\varphi_*\sA)_{B_F}
\qquad
\text{and}
\qquad
\tilde{\sB}_B \subseteq \varphi(\sB_{B_F}).
\]
We therefore get
\begin{align*}
\tilde{\sB}_B \subseteq \varphi(\sB_{B_F}) &\subseteq \varphi(\sA_{B_{F \circ E}} \cdot \sA'_{B_{F \circ E}})
\\
&= \varphi(\sA_{B_{F \circ E}} )\cdot \varphi(\sA'_{B_{F \circ E}})
\\
&\subseteq (\varphi_* \sA)_{B_{F \circ E \circ F}} \cdot (\varphi_* \sA')_{B_{F \circ E\circ F}}.
\end{align*}
\end{proof}

\begin{corollary}
\label{CorollaryAzumayaSubnet}
	Any Azumaya net on $X$ is isomorphic to a tensor factor of a local matrix net.
\end{corollary}

\begin{proof}
If $\sA$ is an Azumaya net on $X$ and $\sA'$ is another net on $X$ such that there exists an isomorphism $\varphi : \sA \otimes \sA' \to \sB$ with a local matrix net $\sB$, then by Prop.~\ref{PropImageOfTensorFactor}, $\varphi_*\sA$ is a tensor factor of $\sB$ and by Remark~\ref{RemarkAisomorphicToImage}, we have $\varphi_*\sA \cong \sA$.
\end{proof}

\begin{comment}
\begin{corollary}
\label{CorollaryImageCoarselyLocal}
Let $\sA$ and $\sB$ be nets on $X$ and let $\varphi : \sA \to \sB$ be a homomorphism.
Then if $\sA$ is coarsely local, so is $\varphi_*\sA$.
\end{corollary}

\begin{proof}
Let $E$ be a locality preserving entourage for $\sA$ and let $Y \subseteq X$, so that $\sA = \sC \otimes \sC'$ with $\sC$ supported on $Y_E$ and $\sC'$ supported on $(X\setminus Y)_E$.
Then by Prop.~\ref{PropImageOfTensorFactor}, 
\[
\varphi_*\sA \cong \varphi_*\sC \otimes \varphi_*\sC'.
\]
If now $\varphi$ is $F$-controlled for an entourage $F$, then $\varphi_*\sC$ is supported in $(Y_E)_F = Y_{E \circ F}$ and $\varphi_*\sC'$ is supported in $(X\setminus Y)_{E \circ F}$. 
This shows that $E \circ F$ is a locality controlling entourage for $\varphi_*\sA$, hence $\varphi_*\sA$ is coarsely local.
\end{proof}
\end{comment}

\begin{proposition}[Nested tensor factors]
\label{PropNestedSequence}
Let $\sA \subseteq \sB \subseteq \sC$ be nets on $X$.
Assume that both $\sA$ and $\sB$ are tensor factors of $\sC$. 
Then $\sA$ is a tensor factor of $\sB$. 
More precisely, if $\sA'$ denotes the commutant of $\sA$ in $\sC$, then the commutant of $\sA$ in $\sB$ is the net $\sA' \cap \sB$, given by
\[
(\sA' \cap \sB)_B = \sA'_B \cap \sB_B.
\]
\end{proposition}

\begin{proof}
Since $\sA$ is a tensor factor of $\sC$, the multiplication map $\sA_X \otimes (\sA_X)' \to \sC_X$ is an isomorphism, and it remains injective when the second factor is restricted to the commutant $(\sA_X)' \cap \sB_X$ of $\sA_X$ in $\sB_X$.
%After tensoring with a tensor complement of $\sC$, we may assume that $\sC$ is a local matrix net.

Since $\sB$ is a tensor factor of $\sC$, Lemma~\ref{LemmaCharacterizationTensorFactor} yields an entourage $E$ such that
\begin{equation}
\label{eq:C-in-A}
\sC_B \subseteq \sB_{B_{E}} \cdot \sB'_{B_{E}}
\end{equation}
for all bounded $B \subseteq X$.
We therefore have the inclusion
\begin{equation}
\label{FurtherInclusion}
\sA'_B = (\sA_X)' \cap \sC_B
\subseteq (\sA_X)' \cap (\sB_{B_E} \cdot \sB'_{B_E}) = \bigl((\sA_X)' \cap \sB_{B_E}\bigr) \cdot \sB'_{B_E}
\end{equation}
To see the last equality, first notice that since $\sB$ is a tensor factor of $\sC$, the multiplication maps $\sB_X \otimes (\sB_X)' \to \sC_X$ is injective.
So after choosing a vector space basis $x_1, \dots, x_n$ for $\sB'_{B_E}$, any element $c \in \sB_{B_E} \cdot \sB'_{B_E} \cong \sB_{B_E} \otimes \sB'_{B_E}$ can be uniquely written as 
\[
c = \sum_{i=1}^n b_ix_i, \qquad b_1, \dots, b_n \in \sB_{B_E}.
\]
Suppose that $c$ is also contained in $(\sA_X)'$, hence commutes with each $a \in \sA_X$. 
Since $\sA_X \subseteq \sB_X$ and $\sB'_{B_E} \subseteq (\sB_X)' \subseteq (\sA_X)'$, every $x_i$ commutes with $a$.
Hence
\[
0 = ac - ca = \sum_{i=1}^n(ab_i - b_i a) x_i.
\]
By uniqueness of the representations, we get $ab_i - b_i a = 0$. Since $a$ was arbitrary, we get that $b_i \in (\sA_X)' \cap \sB_{B_E}$, as claimed.

\begin{comment}
Here is a more general statement.
Let $R$ be a commutative $R$-algebra and let $X$ be an $R$-algebra which splits as a tensor factor 
\end{comment}

By Lemma~\ref{LemmaCharacterizationTensorFactor}, the inclusion \eqref{FurtherInclusion} shows that 
\[
\sA' \cong (\sA' \cap \sB) \otimes \sB', \qquad \text{hence} \qquad \sC \cong \sA \otimes \sA' \cong \sA \otimes (\sA' \cap \sB) \otimes \sB'.
\]
In particular, we get $\sB \cong \sA \otimes (\sA' \cap \sB)$.
\end{proof}

\subsection{QCA and coarsely local automorphisms}
\label{SectionCoarselyLocal}

Let $X$ be a bornological coarse space.
A \emph{quantum cellular automaton}, or \emph{QCA} for short, is just another name for a (controlled) automorphism of a local matrix net on $X$.
QCA are usually considered modulo a special class of automorphisms, called \emph{finite depth quantum circuits}.
However, this notion is not coarsely invariant; instead, we  introduce the new notion of a \emph{coarsely local} automorphism.
In \S\ref{SectionCircuit} below, we show that the two notions agree on spaces of bounded geometry.

We need the following notions:
A collection $(B_i)_{i \in I}$ of subsets of $X$ is called \emph{uniformly bounded} if $\bigcup_{i \in I} B_i \times B_i$ is an entourage.
In particular, each $B_i$ must be bounded.
It is \emph{locally finite} if each $B \subseteq X$ bounded has non-trivial intersection with at most finitely many $B_i$.

\begin{definition}[coarsely local automorphisms]
Let $\sA$ be a local matrix net on $X$.
An automorphism $\alpha$ of $\sA$ is \emph{coarsely local} if
there exists a uniformly bounded and locally finite collection $(B_i)_{i \in I}$ of subsets of $X$, together with a decomposition
\[
\sA \cong \bigotimes_{i \in I} \sA^{(i)}
\]
into pairwise commuting tensor factors $\sA^{(i)}$ of $\sA$ supported on $B_i$, such that $\alpha$ restricts to an automorphism of $\sA^{(i)}$ for each $i \in I$.
We denote by $\LAut(\sA)\subseteq\Aut(\sA)$ the subgroup consisting of finite products of coarsely local automorphisms.
\end{definition}

The infinite tensor product in the above definition is to be understood as
\[
\sA_B \cong \bigotimes_{\substack{i \in I \\ B_i \cap B \neq \emptyset}} \sA^{(i)}_B,
\]
which is a finite tensor product because the family $(B_i)_{i \in I}$ is locally finite.
We require that the isomorphism above is given by multiplication, which makes sense because the subalgebras $\sA^{(i)}_B \subseteq \sA_B$ all pairwise commute.
Note that $\alpha$ is controlled by $\bigcup_{i \in I} B_i \times B_i$.
Any local automorphism is coarsely local; here one may take the tensor decomposition $\sA \cong \bigotimes_{x \in X} \sA|_x$, which is subordinate to the decomposition of $X$ into singleton subsets.

\begin{lemma}
\label{LemmaLAutNormal}
	$\LAut(\sA)$ is a normal subgroup of $\Aut(\sA)$.
\end{lemma}

\begin{proof}
Let $\alpha$ be a coarsely local automorphism subordinate to the tensor factorization $\sA \cong \bigotimes_{i \in I}\sA^{(i)}$.
Then if $\beta \in \Aut(\sA)$ is arbitrary, $\sA \cong \bigotimes_{i \in I} \beta_*\sA^{(i)}$ is another tensor factori\-zation, and $\beta \circ \alpha \circ \beta^{-1}$ restricts to an automorphism of $\beta_*\sA^{(i)}$ for each $i \in I$.
If $\sA^{(i)}$ is supported on $B_i$ and $\beta$ is $E$-controlled, then $\beta_*\sA^{(i)}$ is supported on $(B_i)_E$.
Clearly, the collection $((B_i)_E)_{i \in I}$ is again uniformly bounded.
Since $B \cap (B_i)_E \neq \emptyset$ if and only if $B_{E^{-1}} \cap B_i \neq \emptyset$ and $(B_i)_{i \in I}$ is locally finite by assumption, $((B_i)_E)_{i \in I}$ is also locally finite.
Hence $\beta \circ \alpha \circ \beta^{-1}$ is again coarsely local.
This implies the lemma.
\end{proof}

By the above lemma, we can make the following definition:

\begin{definition}[unstable QCA group]
For a local net $\sA$ on $X$, the group of \emph{quantum cellular automata on $\sA$} is defined by
	\[
	\QCA(\sA) := \Aut(\sA) / \LAut(\sA).
	\]
\end{definition}

We now proceed to define the stable QCA group, where one is allowed to tensor with additional degrees of freedom.

\begin{lemma}
\label{LemmaLocalPushforward}
If $\varphi : \sA \to \sB$ is a \emph{local} isomorphism $X$, we obtain a group isomorphism
\begin{equation}
\label{Connecting2}
\varphi_* : \QCA(\sA) \longrightarrow \QCA(\sB), \qquad [\alpha] \longmapsto [\varphi \circ \alpha \circ \varphi^{-1}].
\end{equation}
For any two local automorphisms $\varphi, \psi : \sA \to \sB$, we have $\varphi_* = \psi_*$.
\end{lemma}

\begin{proof}
The group homomorphism is well-defined because the conjugation of a coarsely local automorphism on $\sA$ by an isomorphism $\varphi : \sA \to \sB$ is a coarsely local automorphism on $\sB$, by an argument similar to that of Lemma~\ref{LemmaLAutNormal}.

Let $Y \subseteq X$ be a support for $\sA$, so that $\sA_B = \bigotimes_{x \in B \cap Y} \sA_x$.
Since $\varphi$ and $\psi$ are local, they induce embeddings $\varphi_x, \psi_x : \sA_x \hookrightarrow \sB_x$.
Since any two embeddings of a complex matrix algebra into another are conjugate, there exist unitaries $u_x \in \sB_x$ such that 
\[
\psi_x = \Ad_{u_x} \circ \varphi_x.
\]
These $(u_x)_{x \in Y}$ define a local automorphism $\Ad_u := \bigotimes_{x \in Y} \Ad_{u_x}$, which is in particular coarsely local.
For any automorphism $\alpha \in \Aut(\sA)$ and any $a \in \sA_X$, we then have
\begin{align*}
(\psi_*\alpha)(\psi(a))  
&= \psi(\alpha(a)) 
\\
&= \Ad_u ( \varphi(\alpha(a)))
\\
&= \Ad_u (\varphi_*\alpha(\varphi(a)))
\\
&= \Ad_u (\varphi_*\alpha(\Ad_u^{-1}(\psi(a)))).
\end{align*}
hence
\[
\psi_* \alpha = \Ad_u \circ (\varphi_*\alpha) \circ \Ad_u^{-1}.
\]
Therefore $\varphi_* \alpha$ and $\psi_*\alpha$ agree in the quotient $\QCA(\sB)$, so $\varphi_* = \psi_*$.
\end{proof}

We now define a preorder on the set\footnote{If the reader is worried about size issues here (the objects of $\Loc(X)$ might be too large to form a set), one may instead only consider nets of the form Example~\ref{ExampleCoordinateNets} in this construction.} of local matrix nets on $X$.
We set $\sA \preceq \sB$ if there exists another local matrix net $\sA'$ and a local (i.e., $\Delta$-controlled) isomorphism $\sA \otimes \sA' \cong \sB$.
If $\sA \preceq \sB$, there exists a canonical group homomorphism
\[
\QCA(\sA) \longrightarrow \QCA(\sA \otimes \sA') \cong \QCA(\sB),
\]
given by stabilization, $[\alpha] \mapsto [\alpha \otimes \id_{\sA'}]$, followed by conjugation with an arbitrary local isomorphism $\varphi : \sA \otimes \sA' \to \sB$.
This does not depend on the choice of $\varphi$ by Lemma~\ref{LemmaLocalPushforward}.

For any two local matrix nets $\sA$ and $\sA'$, we have $\sA\preceq \sA \otimes \sA'$ and $\sA' \preceq \sA \otimes \sA'$,
hence any two elements of our preorder have a common upper bound.
Therefore the associated category with precisely one morphism for each relation $\sA \preceq \sB$ is a filtered category and assigning QCA groups provides a functor from this category to the category of abelian groups.
We can therefore make the following definition:

\begin{definition}[stable QCA group]
Let $X$ be a bornological coarse space.
The group $\QCA(X)$ of \emph{stable quantum cellular automata} on $X$ is given as the filtered colimit
\[
\QCA(X) := \colim_{\sA} \QCA(\sA),
\]
taken over the preorder on the set of local matrix nets on $X$, as defined above.
\end{definition}

\begin{remark}
\label{ExplicitDescriptionQCAgroup}
Using the standard explicit representation for a filtered colimit of groups, the elements of $\QCA(X)$ are given by equivalence classes of pairs $(\sA, \alpha)$ consisting of a local matrix net $\sA$ on $X$ and an automorphism $\alpha$ of $\sA$, where the equivalence relation is generated by the following three relations:
\begin{enumerate}[(1)]
	\item $(\sA, \alpha) \sim (\sA \otimes \sA', \alpha \otimes \id_{\sA'})$ for each local matrix net $\sA'$;
	\item $(\sA, \alpha) \sim (\sA', \varphi \circ \alpha \circ \varphi^{-1})$ for each \emph{local} isomorphism $\varphi: \sA \to \sA'$.
	\item 
	$(\sA, \alpha) \sim (\sA, \id_{\sA})$ for any coarsely local automorphism $\alpha$.
\end{enumerate}
The product of $[\sA, \alpha]$ and $[\sA, \alpha']$ is given by
\[
[\sA, \alpha] \cdot [\sA, \alpha'] = [\sA, \alpha \circ \alpha'].
\]
Two general elements $[\sA, \alpha]$ and $[\sA', \alpha']$ may first be represented by automorphisms on a local matrix net $\sB$ with $\sA, \sA' \preceq \sB$ and then composed in $\Aut(\sB)$.
\end{remark}

\begin{lemma}
\label{LemmaQCAgroupAbelian}
	The group $\QCA(X)$ is abelian.
\end{lemma}

\begin{proof}
There is another associative monoid operation on $\QCA(X)$ given by
\[
[\sA, \alpha] * [\sB, \beta] := [\sA \otimes \sB, \alpha \otimes \beta].
\]
We have
\begin{align*}
\bigl([\sA, \alpha] \cdot [\sA, \alpha']\bigr) * \bigl([\sB, \beta] \cdot [\sB, \beta']\bigr)
&= [\sA, \alpha \circ \alpha'] *[\sB, \beta \circ \beta']
\\
&= [\sA \otimes \sB, (\alpha \circ \alpha') \otimes (\beta \circ \beta')]
\\
&= [\sA \otimes \sB, (\alpha \otimes \beta) \circ (\alpha' \otimes \beta')]
\\
&= [\sA \otimes \sB, \alpha \otimes \beta] \cdot [\sA \otimes \sB, \alpha' \otimes \beta']
\\
&= \bigl([\sA, \alpha] * [\sB, \beta]\bigr) \cdot \bigl([\sA, \alpha'] * [\sB, \beta']\bigr),
\end{align*}
Hence this monoid operation commutes with the group structure.
By the Eckmann-Hilton argument, both operations must agree and be commutative.
\end{proof}

\begin{remark}
If $f : X \to X'$ is a coarse map between bornological coarse spaces, then the image of a uniformly bounded and locally finite collection of subsets is again uniformly bounded and locally finite.
Hence if $\alpha$ is a coarsely local automorphism of a local matrix net $\sA$ on $X$, then $f_* \alpha$ is a coarsely local automorphism of $f_*\sA$.
This shows that $X \mapsto \QCA(X)$ is a functor from the category of bornological coarse spaces to abelian groups. 
\end{remark}

\subsection{Quantum circuits}
\label{SectionCircuit}

Let $X$ be a bornological coarse space. 
The following is a straightforward adaptation to general coarse spaces of the usual notion of quantum circuits in QCA theory:

\begin{definition}[quantum circuits]
Let $\sA$ be a local matrix net on $X$.
An automorphism $\alpha$ of $\sA$ is a \emph{depth one quantum circuit} if there exists a uniformly bounded, locally finite and pairwise disjoint collection $(B_i)_{i \in I}$ of subsets such that $\alpha$ is controlled by the entourage $\bigcup_{i \in I} B_i \times B_i$.
We say that $(B_i)_{i \in I}$ is the \emph{support} of the quantum circuit $\alpha$.
An automorphism $\alpha$ of $\sA$ is a \emph{depth $n$ quantum circuit} if it is a composition of at most $n$ depth one quantum circuits and a \emph{finite depth quantum circuit} if it is a quantum circuit of some finite depth $n$.
By $\Cir(\sA) \subseteq\Aut(\sA)$, we denote the subgroup of finite depth quantum circuits.
\end{definition}

If $\alpha$ is a depth one quantum circuit with support $(B_i)_{i \in I}$, then $\alpha$ restricts to an automorphism of $\sA_{B_i}$ for each $i \in I$ and acts as the identity on $\sA_{X \setminus \bigcup_{i \in I} B_i}$.
Since any automorphism of a matrix algebra is inner, there exist unitaries $u_i \in \sA_{B_i}$ such that
\[
\alpha(a) =  \Bigg(\prod_{\substack{i \in I\\ B_i \cap B \neq \emptyset}} u_i\Bigg) \cdot  a \cdot \Bigg(\prod_{\substack{i \in I\\ B_i \cap B \neq \emptyset}} u_i^*\Bigg)
\]
for each $a \in \sA_B$.
Here each product is finite since the collection $(B_i)_{i \in I}$ is locally finite.

\begin{remark}
Any depth one quantum circuit is a coarsely local automorphism.
\end{remark}

\begin{remark}
There are two interrelated problems with the notion of quantum circuits:
\begin{enumerate}[(1)]
\item 
It is claimed in \cite[Prop.~2.16]{QCAspace} that $\Cir(\sA)$ is a normal subgroup of $\Aut(\sA)$ for general discrete metric spaces.
However, the proof contains a gap that we do not know how to fix without additional hypothesis; in fact, it seems that $\Cir(\sA)$ is not a normal subgroup in general.
To illustrate the problem, suppose that $\alpha$ is a depth one quantum circuit with support $(B_i)_{i \in I}$. 
Then if $\beta \in \Aut(\sA)$ is $E$-controlled, $\beta \circ \alpha \circ \beta^{-1}$ has support $((B_i)_E)_{i \in I}$.
Now, since the fattened collection $((B_i)_E)_{i \in I}$ is no longer pairwise disjoint, one must try to regroup the sets $(B_i)_E$ into finitely many pairwise disjoint subcollections (``layers''), showing that $\beta \circ \alpha \circ \beta^{-1}$ is a finite depth quantum circuit. 
\\
However, this ``distribution into layers'' argument may fail in general.
For example, consider a graph where every vertex has finite degree, but which is not $n$-colorable for any $n$ (for example, the disjoint union $\bigcup_{n \in \N} K_n$, where $K_n$ is the complete graph with $n$ vertices).
Equip the vertex set $X$ with the pseudometric where two vertices have distance $k$ if one needs at least $k$ edges to connect them with a path.
Then the collection $(B_x)_{x \in X}$ of singleton subsets of $X$ is uniformly bounded, locally finite and pairwise disjoint, but its $1$-fattening may not be regrouped into finitely many layers with the same properties.
\item 
The notion of quantum circuits is \emph{not} a coarsely invariant notion.
The problem is that images of pairwise disjoint collections of subsets under coarse maps need not be pairwise disjoint again, and it may not even be possible to divide them into finitely many pairwise disjoint families.
\end{enumerate}
\end{remark}

There are conditions on the underlying space which ensure that the notion of a quantum circuit is nevertheless well-behaved.

\begin{definition}
Let $X$ be a coarse space.
\begin{enumerate}[(1)]
\item
A subset $Y \subseteq X$ is called \emph{uniformly locally finite} if for each entourage $E$, there exists a number $n \in \N$ such that for all $B \subseteq X$ with $B \times B \subseteq E$, we have
\[
\# (B \cap Y) \leq n.
\]
	\item 
	$X$ is said to have \emph{bounded geometry} if it admits a coarsely dense and uniformly locally finite subset $Y \subseteq X$.
	(Here $Y$ is called \emph{coarsely dense} if there exists an entourage such that $Y_E = X$.)
\end{enumerate}
\end{definition}

Of course, basic examples for spaces with bounded geometry are $\Z^n$ or $\R^n$.
However, $\R^n$ is of course not uniformly locally finite and also contains locally finite subsets which are not uniformly locally finite.

\begin{proposition}
\label{PropLocalvsCirc}
If $\sA$ admits a uniformly locally finite support, then we have 
\[
\LAut(\sA) = \Cir(\sA).
\]
\end{proposition}

In particular, it follows from Lemma~\ref{LemmaLAutNormal} above that under the above support assumption on $\sA$, $\Cir(\sA)$ is a normal subgroup of $\Aut(\sA)$.

\begin{proof}
We always have the inclusion $\Cir(\sA) \subseteq \LAut(\sA)$, so it remains to show the converse inclusion.
We start with some preliminary observations.
For an entourage $E$ of $X$, write 
\[
E[x] := \{ y \in X \mid (y, x) \in E\}.
\]
Observe that $E[x] \times E[x] \subseteq E^{-1} \circ E$, hence the collection $(E[x])_{x \in X}$ is uniformly bounded.
Morever, if $Y \subseteq X$ is a uniformly locally finite support for $\sA$, the collection $(E[y] \cap Y)_{y \in X}$ is uniformly bounded and locally finite.

Let $n \in \N$ be such that for any $B \subseteq X$ with $B \times B \subseteq E \circ E^{-1}$, we have $\#(B \cap Y) \leq n$.
Then each $x \in Y$ is contained in $E[y]$ for at most $n$ points $y \in Y$: 
Indeed, we have $x \in E[y]$ if and only if $y \in E^{-1}[x]$, and $E^{-1}[x] \times E^{-1}[x] \subseteq E \circ E^{-1}$, hence $\#(E^{-1}[x] \cap Y) \leq n$ by choice of $n$.
Therefore, after well-ordering $Y$, we obtain a decomposition $Y = Y_0 \sqcup \cdots \sqcup Y_n$ such that the collections $(E[y] \cap Y)_{y \in Y_k}$, $k=0, \dots, n$ are each pairwise disjoint.

Let now $\alpha$ be a coarsely local automorphism of $\sA$ and let $\sA \cong \bigotimes_{i \in I} \sA^{(i)}$ be a tensor decomposition as in the definition of coarse locality, where $\sA^{(i)}$ is supported on the member $B_i$ of a uniformly bounded and locally finite collection $(B_i)_{i \in I}$.
We assume that $\sA^{(i)}$ is non-trivial, so that $B_i \cap Y \neq \emptyset$ for each $i \in I$.
We now use the considerations above for the entourage
\[
E := \bigcup_{i \in I} B_i \times B_i.
\]
Here we have
\[
E[y] = \{x \in X \mid \exists i \in I: x, y \in B_i\} = \bigcup_{\substack{i \in I \\ y \in B_i}} B_i.
\]
Hence each $B_i$ is contained in some $E[y]$.
We may therefore choose for each $B_i$ an element $y_i \in Y$ such that $B_i \subseteq E[y_i]$.
This gives a decomposition $I = I_0 \sqcup \cdots \sqcup I_n$, where $I_k = \{i \in I \mid y_i \in Y_k\}$.
Moreover, this results in a decomposition $\alpha = \alpha_0 \circ \cdots \circ \alpha_n$, where $\alpha_k$ acts as $\alpha$ on $\bigotimes_{i \in I_k} \sA^{(i)}$ and trivial on $\bigotimes_{i \in I \setminus I_k} \sA^{(i)}$.
By construction, each automorphism $\alpha_k$ is a depth one quantum circuit with support $(E[y] \cap Y)_{y \in Y_k}$.
\end{proof}

\begin{remark}
If $X$ has bounded geometry, then any local matrix net $\sA$ on $X$ is isomorphic to one allowing a uniformly locally bounded support.
Indeed, let $\sA$ be an arbitrary local matrix net on $X$ and let $Y \subseteq X$ be a coarsely dense and uniformly locally finite subset.
Let $Z \subseteq X$ be a locally finite support for $\sA$, which exists by Remark~\ref{RemarkLocallyFiniteSupport}.
Let $E$ be an entourage such that $Y_E = X$.
Then for each point $z \in Z$, there exists a point $y \in Y$ such that $(z, y) \in E$.
Hence we may choose a map $f : Z \to Y$ that is close to the identity (hence coarse).
It is proper by local finiteness of $Y$ and $Z$.
By Prop.~\ref{PropCloseness}, the local matrix net $f_*\sA$ is therefore isomorphic to $\sA$ and has support $Y$ by construction.

Together with Prop.~\ref{PropLocalvsCirc}, this shows that in the case that $X$ has bounded geometry, the group $\QCA(X)$ may equivalently be presented by pairs $(\sA, \alpha)$, where $\sA$ is a local matrix net admitting a uniformly locally finite support and $\alpha$ is a (controlled) automorphism of $\sA$.
The relations are then the same as in Remark~\ref{ExplicitDescriptionQCAgroup}, except that one may declare that all quantum circuits are trivial instead of coarsely local automorphisms.
In particular, for $X = \Z^n$, $\QCA(\Z^n)$ agrees with the usual QCA group.
\end{remark}

\section{A coarse homology theory}

In this section, we construct a coarse homology theory whose degree zero groups are given by quantum cellular automata.

Let us fix set-theoretic size issues for the constructions in this section: We choose two Grothendieck universes, whose elements are called \emph{small sets} and \emph{large sets}.
All bornological coarse spaces and nets on these are assumed to have underlying sets coming from the first Grothendieck universe.
The set of all nets on a bornological coarse space will then form a large set.

\subsection{QCA and $K$-theory}

If $\mathfrak{C}$ is a symmetric monoidal category, its $K$-theory group $K_0(\mathfrak{C})$ is the abelian group obtained by group completing the commutative monoid of isomorphism classes of objects of $\mathfrak{C}$, with monoid operation given by the tensor product of $\mathfrak{C}$,
\[
[x] + [y] := [x \otimes y].
\]

To define the higher $K$-theory groups of $\mathfrak{C}$, one uses the classifying space functor $B : \Cat \to \Spaces$, the latter denoting the $\infty$-category of (large) simplicial sets.
This functor preserves commutative monoid objects: If $\mathfrak{C}$ is a symmetric monoidal category, then $B \mathfrak{C}$ is a commutative monoid in $\Spaces$.
Here one may form the $\infty$-categorical group completion functor $\CMon(\Spaces) \to \CMon^{\mathrm{grp}}(\Spaces)$ to obtain a grouplike commutative monoid in spaces, which by May's recognition principle is the same information as an infinite loop space, i.e., a connective spectrum.
In total, we get a $K$-theory functor
\[
  K : \CMon(\Cat) \xrightarrow{~~B~~} \CMon(\Spaces) \xrightarrow{~~\substack{\text{group} \\ \text{completion}}~~} \CMon^{\mathrm{grp}}(\Spaces) \simeq \Spectra_{\geq 0} 
\]
from symmetric monoidal categories to connective spectra.
The functor $B$ commutes with colimits, and group completion is a left adjoint, so $K$ also preserves colimits.

\begin{remark}
\label{RemarkDescriptionKgroups}
The $K$-theory functor has the property that
\[
\pi_0(K(\mathfrak{C})) = K_0(\mathfrak{C}),
\]
the $K$-theory group described above and one defines the higher $K$-groups of $\mathfrak{C}$ as the higher homotopy groups of $K(\mathfrak{C})$.
If $\mathfrak{C}$ is a groupoid, the first $K$-theory group $K_1(\mathfrak{C}) = \pi_1(K(\mathfrak{C}))$ can be described as follows:
Its elements are equivalence classes of pairs $(x, \alpha)$, where $x$ is an object of $\mathfrak{C}$ and $\alpha$ is an automorphism of $x$.
The equivalence relation is generated by
\begin{enumerate}[(1)]
	\item $(x, \alpha) \sim (x \otimes y, \alpha \otimes \id_y)$ for any object $y$ of $\mathfrak{C}$;
	\item $(x, \alpha) \sim (x', \varphi \circ \alpha \circ \varphi^{-1})$ for each isomorphism $\varphi : x \to x'$.
\end{enumerate}
The group operation is given by tensor product,
\[
[x, \alpha] + [y, \beta] := [x \otimes y, \alpha \otimes \beta],
\]
or, alternatively, by representing two given elements of $K_1(\mathfrak{C})$ by automorphisms of the same object $x$ of $\mathfrak{C}$ and then taking the product of $\Aut(x)$.
\end{remark}

Applying the $K$-theory functor to the various categories of nets, we obtain various functors from the category of bornological coarse spaces to connective spectra.
By definition, the category $\Loc(X)$ is \emph{monoidally cofinal} in $\Az(X)$, meaning that each object of $\Az(X)$ is a tensor factor of an object of $\Loc(X)$.
By the explicit description of $K_1$ from Remark~\ref{RemarkDescriptionKgroups}, we therefore get that inclusion induces an isomorphism
\begin{equation}
\label{K1Iso}
K_1(\Loc(X)) \cong K_1(\Az(X)).
\end{equation}
By standard properties of the $K$-theory functor, we also have an isomorphism of the higher homotopy groups, see Lemma~\ref{LemmaHigherHomotopyGroups}.

The relation to quantum cellular automata is given by the following result, which was essentially shown in \cite{QCAspace}, using properties of Quillen's plus construction.
We give a different, more elementary proof.

\begin{theorem}
\label{ThmK1QCA}
For any bornological coarse space $X$, there is a natural group isomorphism
\[
\QCA(X) \cong K_1(\Az(X)).
\]
\end{theorem}

\begin{proof}
By \eqref{K1Iso}, we have $K_1(\Loc(X)) \cong K_1(\Az(X))$, so we may equivalently construct a homomorphism to $K_1(\Loc(X))$.
By the generators and relations description of both groups from Remarks~\ref{RemarkDescriptionKgroups}, respectively \ref{ExplicitDescriptionQCAgroup}, we observe that both groups are given by pairs $(\sA, \alpha)$ of a local matrix net and an automorphism $\alpha$ of $\sA$.
We claim that the map
\[
\QCA(X) \longrightarrow K_1(\Loc(X)), \qquad [\sA, \alpha] \longmapsto [\sA, \alpha]
\]
is well-defined, for which we have to check that the relations defining the QCA group also hold in the $K_1$ group.

In both groups, we have the stabilization relations $(\sA, \alpha) \sim (\sA \otimes \sA', \alpha \otimes \id_{\sA'})$.
The second relation in $\QCA(X)$ is $(\sA, \alpha) \sim (\sA', \varphi \circ \alpha \circ \varphi^{-1})$ for local isomorphisms $\varphi$, which also holds in $K_1$ (even for \emph{all} isomorphisms).
The last relation of $\QCA(X)$ is that $[\sA, \alpha] = 0$ when $\alpha$ is a coarsely local automorphism.
Let $\sA \cong \bigotimes_{i \in I} \sA^{(i)}$ be a tensor decomposition as in the definition of a coarsely local automorphism, where $\sA^{(i)}$ is supported on the member $B_i$ of a uniformly bounded and locally finite collection $(B_i)_{i \in I}$ of subsets of $X$ and $\alpha$ restricts to an automorphism $\alpha_i$ of $\sA^{(i)}$ for each $i \in I$.
Because we have an isomorphism
\[
\sA_X \cong \bigotimes_{i \in I} \sA^{(i)}_{B_i}
\]
with the right hand side a tensor product of matrix algebras, each $\sA_{B_i}^{(i)}$ must be a matrix algebra, so the automorphism $\alpha_i$ of $\smash{\sA^{(i)}_{B_i}}$ is implemented by a unitary, unique up to phase.
As the projective unitary group $PU(n)$ is simple for $n\geq 2$, it follows from \cite{Goto} that each of its elements may be written as a multiplicative commutator.
Hence for each $i \in I$, we have $\alpha_i = \Ad_{u_i} \Ad_{v_i} \Ad_{u_i}^{-1} \Ad_{v_i}^{-1}$ for unitaries $u_i, v_i \in \smash{\sA^{(i)}_{B_i}}$ and globally, $\alpha$ is the multiplicative commutator of 
\[
\Ad_u = \bigotimes_{i \in I} \Ad_{u_i} \qquad \text{and} \qquad \Ad_v = \bigotimes_{i \in I} \Ad_{v_i},
\]
which are controlled since $(B_i)_{i \in I}$ is uniformly bounded. 
We obtain that $[\sA, \alpha]$ is a commutator, hence vanishes as $K_1(\Loc(X))$ is abelian. 
This shows that the map is well defined.
It is obvious from the definition of the group structures on both sides that the map is a group homomorphism.
It is also clear that it is natural in $X$.

It is clear that the group homomorphism is surjective.
To see that it is injective, we have to show that the second relation of $K_1(\Loc(X))$ also holds in $\QCA(X)$, which is $(\sA, \alpha) \sim (\sA', \varphi \circ \alpha \circ \varphi^{-1})$ whenever $\varphi : \sA \to \sA'$ is an isomorphism between local matrix nets $\sA$ and $\sA'$.
To this end, consider the automorphism $\Phi$ of $\sA \otimes \sA'$ given by
\[
\Phi(a \otimes a') := \varphi^{-1}(a') \otimes \varphi(a).
\]
Since $\Phi^{-1} = \Phi$, we have
\begin{align*}
\bigl(\Phi \circ (\alpha \otimes \id_{\sA'}) \circ \Phi^{-1}\bigr)
(a \otimes a')
&= \bigl(\Phi \circ (\alpha \otimes \id_{\sA'})\bigr)\bigl(\varphi^{-1}(a') \otimes \varphi(a)\bigr)
\\
&= \Phi \bigl((\alpha \circ \varphi^{-1})(a') \otimes \varphi(a)\bigr)
\\
&= a \otimes (\varphi \circ \alpha \circ \varphi^{-1})(a'), 
\end{align*}
hence combined with the first relation, the second relation of $K_1(\Loc(X))$ is equivalent to the relation that $(\sA, \alpha) \sim (\sA, \varphi \circ \alpha \circ \varphi^{-1})$ whenever $\varphi$ is an \emph{auto}morphism of $\sA$.
But by Lemma~\ref{LemmaQCAgroupAbelian}, $\QCA(X)$ is abelian, so the same relation holds here.
This shows injectivity.
\end{proof}

\begin{theorem}
\label{ThmFlasqueContractible}
If $X$ is flasque, then $K(\Az(X)) \simeq 0$.
\end{theorem}

\begin{proof}
Because $X$ is flasque, by Prop.~\ref{PropFlasque} there exists an endofunctor $S : \Az(X) \to \Az(X)$ such that $S \cong \id \otimes S$.
Since tensor product corresponds to addition in $K$-theory, this yields a self-map of spectra
\[
S_* : K(\Az(X)) \to K(\Az(X))
\]
with
\[
S_* \simeq \id + S_*.
\]
This implies $\id \simeq 0$, hence $K(\Az(X)) \simeq 0$.
\end{proof}

The following result is more difficult to prove; the proof will be carried out in the next section.
In the statement below, $X \otimes \R$ is the bornological coarse space with underlying set $X \times \R$ and the product coarse structure and bornology, see Example~\ref{TensorProductBornCoarse}.

\begin{theorem}
\label{ThmLoopRn}
For each bornological coarse space $X$, we have a canonical map of spectra
\[
K(\Az(X)) \longrightarrow \Omega K(\Az(X \otimes \R)),
\]
which induces an isomorphism on homotopy groups in non-negative degrees.
\end{theorem}

\subsection{The Mayer-Vietoris axiom}
\label{SectionMayerVietoris}

Following \cite[\S4]{QCAspace}, we use the double mapping cylinder construction of Thomason, see \cite{ThomasonOnline,Thomason1982}.
Suppose we are given a span of symmetric monoidal categories and symmetric monoidal functors
\begin{equation}
\label{squareInvolvingP}
\begin{tikzcd}
C \ar[r, "\iota_1"] \ar[d, "\iota_2"'] & D_1 \ar[d, dashed]
\\
D_2 \ar[r, dashed] & P
\end{tikzcd}
\end{equation}
such that $C$ is a groupoid.
The \emph{double mapping
cylinder} of this span is the symmetric monoidal category $P$ defined as follows.
\begin{enumerate}[(1)]
\item Objects are triples $(c, d_1, d_2)$ with $c$ an object of $C$ and $d_i$ objects of $D_i$.
\item A morphism $(c, d_1, d_2) \to (c', d'_1, d'_2)$ is an equivalence class of quintuples $(f, f_1, f_2, c_1, c_2)$, where $c_1, c_2$ are objects of $C$, $f : c \to c_1 \otimes c' \otimes c_2$ is an isomorphism in $C$ and $f_i : \iota_i(c_i) \otimes d_i \to d'_i$ are morphisms in $D_i$.
Two such tuples are equivalent if they differ  by replacing the objects $c_1$ and $c_2$ with isomorphic ones.
\end{enumerate}

The category $P$ receives symmetric monoidal functors from $D_1$ and $D_2$ in the obvious way, and there is a natural transformation witnessing the commutativity of the square \eqref{squareInvolvingP}.
The significance of this construction is that after applying the $K$-theory functor, one obtains a pushout diagram of spectra
\begin{equation}
\label{ThomasonHomotopyPullback}
\begin{tikzcd}
K(C) \ar[r]  \ar[d] & K(D_1) \ar[d]
\\
K(D_2) \ar[r]& K(P).
\end{tikzcd}
\end{equation}

Let $X$ be a bornological coarse space.

\begin{definition}
For a subset $Y \subseteq X$, the \emph{restriction} $\sA|_Y$ of a net $\sA$ on $X$ to $Y$ is the net on $X$ given by
\[
(\sA|_Y)_B := \sA_{Y \cap B}.
\]
For a big family $\cW$ in a bornological coarse space $X$, denote by 
\[
\Az_{\cW}(X) \subseteq \Az(X)
\]
the full subcategory of those Azumaya nets that are a tensor product of a local net on $X$ and an Azumaya net supported on $\cW$.
Put differently,  $\Az_{\cW}(X)$ consists of those nets $\sA$ on $X$ for which there exists a member $W$ of $\cW$ such that we have a tensor decomposition 
 \[
 \sA = \sA|_W \otimes \sA|_{X \setminus W},
 \]
  where $\sA|_W$ is an Azumaya net and $\sA|_{X \setminus W}$ is a local matrix net.
For a net $\sA$ with these properties, we say that $\sA$ is \emph{local away from $W$}.
 If $\cV$ is another big family with $\cW \subseteq \cY$, we denote by $\Az_{\cW}(\cV)$ the full subcategory of $\Az_{\cW}(X)$ consisting of those nets that are supported on some member of $\cV$.
\end{definition}

\begin{lemma}
\label{LemmaHigherHomotopyGroups}
For any pair of big families $\cW \subseteq \cV$, the inclusions induce isomorphisms
\[
K_n(\Loc(\cV)) \cong K_n(\Az_{\cW}(\cV)) \cong K_n(\Az(\cV))
\]
in all degree $n \geq 1$.
\end{lemma}

\begin{proof}
Both $\Az(\cV)$ and $\Az_{\cW}(\cV)$ are monoidally cofinal in $\Loc(\cV)$, meaning that each object of $\Az(\cV)$, respectively $\Az_{\cW}(\cV)$, is a tensor factor of an object of $\Loc(\cV)$.
This implies that the inclusion induces an isomorphism on the higher $K$-theory groups.
\end{proof}

For two big families $\cY$, $\cZ$ in $X$, we then get a commutative diagram
\begin{equation}
\label{SquareLocalized}
\begin{tikzcd}
  \Az(\cY \Cap \cZ) \ar[r]
  \ar[d]
  & 
  \Az_{\cY \Cap \cZ}(\cY)
  \ar[d]
  \\
  \Az_{\cY \Cap \cZ}(\cZ) \ar[r] &
  \Az_{\cY \Cap \cZ}(\cY \Cup \cZ).
\end{tikzcd}
\end{equation}
In the following, let $P$ be the double mapping cylinder for the top left span in the above diagram.
There is an obvious symmetric monoidal functor
\[
T : P \longrightarrow \Az_{\cY \Cap \cZ}(\cY \Cup \cZ), \qquad (\sA, \sB^{\cY}, \sB^{\cZ}) \longrightarrow \sA\otimes \sB^{\cY} \otimes \sB^{\cZ},
\]
which is easily seen to be faithful using the equivalence relation on morphisms in $P$.
We then have the following result: 

\begin{proposition}
\label{PropSupportedVersion}
The induced map $K(T) : K(P) \to K(\Az_{\cY \Cap \cZ}(\cY \Cup \cZ))$ is a homotopy equivalence.
Consequently, the diagram
\[
\begin{tikzcd}
  K(\Az(\cY \Cap \cZ)) \ar[r]
  \ar[d]
  & 
  K(\Az_{\cY \Cap \cZ}(\cY))
  \ar[d]
  \\
  K(\Az_{\cY \Cap \cZ}(\cZ)) \ar[r] &
  K(\Az_{\cY \Cap \cZ}(\cY \Cup \cZ)).
\end{tikzcd}
\]
is a cartesian square of spectra.
\end{proposition}

Since by the properties of the double mapping cylinder construction, the square \eqref{ThomasonHomotopyPullback} involving $P$ is cartesian, we see that for the proof of Prop.~\ref{PropSupportedVersion}, it suffices to show that $K(T)$ is a homotopy equivalence.

Quillen's theorem A \cite{Quillen1973HigherK} states that this is the case provided that we check that the classifying spaces of the comma categories $\sC \downarrow T$ are weakly contractible for every object $ \sC$. 
Recall that the objects of $\sC \downarrow T$ are morphisms in $\Az_{\cY \Cap \cZ}(\cY \Cup \cZ)$ of the form
	\begin{equation}
	\label{ObjectCommaCategory}
	\varphi : \sC \to T(\sA, \sB^{\cY}, \sB^{\cZ})
	\end{equation}
 and the morphisms of the comma category are commutative diagrams
\[
\begin{tikzcd}[row sep=0.4cm]
& T(\sA, \sB^{\cY}, \sB^{\cZ}) \ar[dd, "T(\psi{,} \psi_{\cY}{,} \psi_{\cZ}{,} \sA^{\cY}{,} \sA^{\cZ})"]
\\
\sC \ar[ur, "\varphi"] \ar[dr, "\tilde{\varphi}"']
\\
& T(\tilde{\sA}, \tilde{\sB}^{\cY}, \tilde{\sB}^{\cZ}).
\end{tikzcd}
\]

\begin{lemma}
\label{ContractibilityLemma}
For every object $\sC$ in $\Az_{\cY \Cap \cZ}(\cY \Cup \cZ)$, the classifying space of the comma category $\sC \downarrow T$ is weakly contractible.
\end{lemma}

\begin{proof}
Fix an object $\sC$ of $\Az_{\cY \Cap \cZ}(\cY \Cup \cZ)$.
A standard way to establish this is that $\sC \downarrow T$ admits an initial object.
However, this does not seem to be the case, so we will instead show the weaker statement that $\sC \downarrow T$ is a filtered union of subcategories, each of which do admit an initial object.
This suffices as then $\sC \downarrow T$ is a filtered union of contractible spaces, hence contractible.

As preparation, first notice that by the support conditions on $\sC$, there exist members $Y \in \cY$ and $Z \in \cZ$ such that $\sC$ is supported on $Y \cup Z$ and local away from $Y \cap Z$.
Hence $(\sC|_{Y \cap Z}, \sC|_{Y \setminus Z}, \sC|_{Z \setminus Y})$ is an object of $P$ and
\[
\sC \cong \sC|_{Y \cap Z} \otimes \sC|_{Y \setminus Z} \otimes \sC|_{Z \setminus Y} = T(\sC|_{Y \cap Z},  \sC|_{Y \setminus Z},  \sC|_{Z \setminus Y}),
\]
defines an object $\texttt{split}_{Y, Z}(\sC)$ of $\sC \downarrow T$.

For every entourage $E$ and any two members $Y \in \cY$, $Z \in \cZ$ such that $\sC$ is supported on $Y \cup Z$ and local away from $Y \cap Z$, we consider the full subcategory  $(\sC \downarrow T)_{Y, Z, E}$ of $\sC \downarrow T$ consisting of those objects of the form \eqref{ObjectCommaCategory} such that 
\begin{enumerate}[$\bullet$]
	\item $\varphi$ is $E$-controlled;
	\item $\sA$ is supported on $Y \cap Z$;
	\item $\sB^{\cY}$ is supported on $Y$;
	\item $\sB^{\cZ}$ is supported on $Z$.
\end{enumerate}

It is clear that $\sC \downarrow T$ is the filtered union of the subcategories $(\sC \downarrow T)_{Y, Z, E}$.
We claim that $\texttt{split}_{Y_E, Z_E}(\sC)$ is an initial object of $(\sC \downarrow T)_{Y, Z, E}$.
Since the functor $T$ is faithful and all net categories are groupoids, we see that there is at most one morphism between any two given objects, determined as the pre-image under $T$ of $\tilde{\varphi} \circ \varphi^{-1}$.
Therefore, it suffices to establish that any object of $(\sC \downarrow T)_{Y, Z, E}$ admits a morphism from $\texttt{split}_{Y_E, Z_E}(\sC)$; uniqueness is then automatic.

To show the claim, we take an arbitrary object 
\[
\varphi : \sC \longrightarrow T(\sA, \sB^{\cY}, \sB^{\cZ}) = \sA \otimes \sB^{\cY} \otimes \sB^{\cZ} =: \sT,
\]
of $(\sC \downarrow T)_{Y, Z, E}$.
Prop.~\ref{PropImageOfTensorFactor} implies that its image $\varphi_*(\sC|_{Y_E \setminus Z_E})$ is a tensor factor of $\sT$. 
Moreover, since $E$ is a control for $\varphi$, the net $\varphi_*( \sC|_{Y_E \setminus Z_E})$ is supported in $X \setminus Z$, which is disjoint from the supports of $\sA$ and $\sB^{\cZ}$.
Therefore
\[
\varphi_*(\sC|_{Y_E \setminus Z_E}) \subseteq \sB^{\cY} \subseteq \sT,
\]
where the first two are tensor factors in $\sT$.
From Prop.~\ref{PropNestedSequence}, we therefore get that $\varphi_*(\sC|_{Y_E \setminus Z_E})$ is also a tensor factor of $\sB^{\cY}$, so
\[
\sB^{\cY} = \varphi_*(\sC|_{Y_E \setminus Z_E}) \otimes \sA^{\cY} 
\]
for some Azumaya net $\sA^{\cY}$ supported on $Y$.
Since $\sA^{\cY}$ is the commutant of $\varphi_*(\sC|_{Y_E \setminus Z_E})$ in $\sB^{\cY}$, it must lie in the image of the commutant $\sC|_{Z_E}$ of $\sC|_{Y_E \setminus Z_E}$ in $\sC$, hence because $\varphi$ is $E$-controlled, we have that $\sA^{\cY}$ is supported on $Z_{E \circ E}$.
In total, $\sA^{\cY}$ is therefore supported on $\cY \Cap \cZ$.
Similarly, we obtain another Azumaya net $\sA^{\cZ}$ supported on $\cY \Cap \cZ$ such that
\[
\sB^{\cZ} = \varphi_*(\sC|_{Z_E \setminus Y_E}) \otimes \sA^{\cZ}.
\]
In total, we get isomorphisms of nets
\[
\psi: \sC|_{Y_E \cap Z_E} \longrightarrow \sA \otimes \sA^{\cY} \otimes \sA^{\cZ}, \qquad \text{and} \qquad 
\begin{aligned}
\psi^{\cY}: \sC|_{Y_E \setminus Z_E} \otimes \sA^{\cY} &\longrightarrow \sB^{\cY}, \\ 
\psi^{\cZ}: \sC|_{Z_E \setminus Y_E} \otimes \sA^{\cZ} &\longrightarrow \sB^{\cZ}, 
\end{aligned}
\]
which define a morphism in $P$ and further a map in $(\sC\downarrow T)_{Y, Z, E}$ from $\texttt{split}_{Y_E, Z_E}(\sC)$ to our given $\varphi : \sC \to T(\sA, \sB^{\cY}, \sB^{\cZ})$.
This proves the claim and finishes the proof.
\end{proof}

Thm.~\ref{ThmFlasqueContractible} generalizes as follows.

\begin{lemma}
\label{LemmaFlasqueExtended}
If $\cY$ is flasque, then $K(\Az_{\cY \Cap \cZ}(\cY))$ is contractible. 
\end{lemma}

\begin{proof}
After applying Lemma~\ref{LemmaHigherHomotopyGroups}, it remains to show that $K_0(\Az_{\cY \Cap \cZ}(\cY)) = 0$.
To this end, consider the functor $S$ from Prop.~\ref{PropFlasque}.
If $\sA$ is an object of $\Az_{\cY \Cap \cZ}(\cY)$ supported on some flasque member $Y$ of $\cY$, we may choose an object $\sA'$ in $\Az_{\cY \Cap \cZ}(\cY)$ such that $\sA \otimes \sA'$ is local. 
Then if $f : Y \to Y$ is the map implementing flasqueness, then 
\[
S(\sA \otimes \sA') = \sA \otimes \sA' \otimes f_* \sA \otimes f_* \sA' \otimes f^2_* \sA \otimes f_*^2 \sA' \otimes \cdots.
\]
is a local matrix net on $Y$, hence a well-defined object of $\Az_{\cY \Cap \cZ}(\cY)$.
We then have an isomorphism of nets
\[
S(\sA \otimes \sA') \cong S(\sA) \otimes S(\sA') \cong \sA \otimes S(\sA) \otimes S(\sA')  \cong \sA \otimes S(\sA \otimes \sA'),
\]
given by shifting the copies of $\sA$.
In $K_0$, this implies 
\[
[S(\sA \otimes \sA')] = [\sA]  + [S(\sA \otimes \sA')] \qquad
\Longrightarrow \qquad [\sA] = 0.
\]
\end{proof}

\begin{corollary}
\label{CorollaryFlasque}
If $\cY$ and $\cZ$ are flasque big families on a bornological coarse space, then we have a canonical homotopy equivalence
\[
K(\Az(\cY \Cap \cZ)) \simeq \Omega K(\Az_{\cY \Cap \cZ}(\cY \Cup \cZ))
\]
\end{corollary}

\begin{proof}
By Lemma~\ref{LemmaFlasqueExtended}, the flasqueness assumption on $\cY$ and $\cZ$ imply that the pushout diagram from Prop.~\ref{PropSupportedVersion} becomes 
\[
\begin{tikzcd}
  K(\Az(\cY \Cap \cZ)) \ar[r]
  \ar[d]
  & 
  *
  \ar[d]
  \\
  * \ar[r] &
  K(\Az_{\cY \Cap \cZ}(\cY \Cup \cZ)),
\end{tikzcd}
\]
which establishes $K(\Az(\cY \Cap \cZ))$ as the loop space of $K(\Az_{\cY \Cap \cZ}(\cY \Cup \cZ))$ (as the $\infty$-category of spectra is stable, pushout and pullback diagrams coincide). 
\end{proof}

\begin{proof}[of Thm.~\ref{ThmLoopRn}]
Consider the big families $\cY := \{X \otimes \R_{\leq 0}\}$ and $\cZ := \{X \otimes \R_{\geq 0}\}$ in $X \otimes \R$.
Their elementwise intersection is the big family generated by $X \otimes \{0\}$, hence by Corollary~\ref{CorollaryEquivalenceBigFamily} and the fact that the $K$-theory functor commutes with filtered colimits, the canonical map
\[
K(\Az(X)) \cong K(\Az(X \otimes \{0\})) \longrightarrow K(\Az(\cY \Cap \cZ))
\]
is a homotopy equivalence.
Since $X \otimes \R_{\leq 0}$ and $X \otimes \R_{\geq 0}$ are flasque, Corollary~\ref{CorollaryFlasque} yields a canonical map
\[
K(\Az(X))\simeq K(\Az(\cY \Cap \cZ)) \simeq \Omega K(\Az_{\cY \Cap \cZ}(\cY \Cup \cZ)) \longrightarrow \Omega K(\Az(X \otimes \R)).
\]
By Lemma~\ref{LemmaHigherHomotopyGroups}, it induces an isomorphism on homotopy groups of nonnegative degree.
\end{proof}

\subsection{Definition of the coarse homology theory}

The notion of group-valued coarse homology theories was first axiomatically introduced by Mitchener \cite{Mitchener2002CoarseHomology}.
The modern point of view is that a group-valued homology theory should be the ``shadow'' of a functor valued in spectra (or some other stable $\infty$-category), which yields a group-valued functor by taking homotopy groups. 
The following definition is equivalent to that of Bunke--Engel \cite[Definition 4.22]{BunkeEngel2020}.

\begin{definition}
A functor
\[
F : \BornCoarse \longrightarrow \Spectra
\]
is a \emph{coarse homology theory} if it satisfies the following axioms:
\begin{enumerate}[(1)]
\item 
$F$ is \emph{coarsely invariant}, i.e., whenever $f, g : X \to X'$ are close, then $F(f) \simeq F(g)$.
\item
$F$ \emph{vanishes on flasques}, i.e., whenever $X$ is flasque, then $F(X) \simeq 0$.
\item
$F$ satisfies the \emph{Mayer-Vietoris axiom}:
Whenever $\cY$ and $\cZ$ are two big families in a coarse space $X$, then we have a pushout square
\[
\begin{tikzcd}
	F(\cY \Cap \cZ) \ar[d] \ar[r] & F(\cY) \ar[d]
	\\
	F(\cZ) \ar[r] & F(\cY \Cup \cZ).
\end{tikzcd}
\]
Here for any big family in $X$, we define $F(\cY) := \colim_{Y \in \cY} F(Y)$.
\item
$F$ is \emph{$u$-continuous}.
This means that if for an entourage $E$ on a bornological coarse space $X$, $X_E$ denotes the space $X$ with the same bornology but equipped with the smallest coarse structure containing $E$, then
\[
F(X) = \colim_{E \in \cC} F(X_E).
\]
\end{enumerate}
\end{definition}

Taking $F$ to be one of the functors $K\circ \Az$ of $K \circ \Loc$, we have already checked axioms (1) and (2) and axiom (4) is also easy to check.
However, neither functor satisfies the Mayer-Vietoris axiom.
The problem is that all the $K$-theory spectra are connective, which generates a problem at the degree zero level.
Fortunately, non-connective deloopings are provided by Thm.~\ref{ThmLoopRn}.

\begin{definition}
For a bornological coarse space $X$, we define 
\[
Q(X) := \colim_{n \to \infty} \Omega^{n+1} K\bigl(\Az(X \otimes \R^n)\bigr),
\]
where the colimit is taken in the $\infty$-category of spectra along the maps provided by Thm.~\ref{ThmLoopRn}.
Varying $X$, this defines a functor
\[
Q : \BornCoarse \longrightarrow \Spectra.
\]
\end{definition}

Since all the spectrum maps from Thm.~\ref{ThmLoopRn} induce homotopy equivalences in positive degree, they particularly induce homotopy equivalences on the corresponding infinite loop spaces, hence the spectrum $Q(X)$ may equivalently be described as the $\Omega$-spectrum $E_0, E_1, \dots$, where the $n$-th space $E_n$ is given by 
\[
E_n = \Omega^{\infty+1} K\bigl(\Az(X \otimes \R^n)\bigr).
\]
Its homotopy groups are given by the formula
\[
Q_n(X) := \pi_n(Q(X)) = K_{n+m+1}\bigl(\Az(X \otimes \R^m)\bigr), \qquad \text{for}~~m \geq -n-1.
\]
The definition is made so that for any bornological coarse space $X$, we have
\[
Q_0(X) = K_1(\Az(X)) \cong \QCA(X),
\]
using the isomorphism from Thm.~\ref{ThmK1QCA}.
Our main result is now the following.

\begin{theorem}
The functor $Q$ is a coarse homology theory.
\end{theorem}

\begin{proof}
	Coarse invariance follows from Prop.~\ref{PropCloseness} and the properties of the $K$-theory functor.
	That $Q$ vanishes on flasques follows from Thm.~\ref{ThmFlasqueContractible}.
Both these properties hold already before taking the colimit.
	
Since $Q(X)$ is represented by a $\Omega$-spectrum with spaces $\Omega^{\infty +n+1} K(\Az(X \otimes \R^n))$, it suffices that at each level the corresponding square of spaces is a homotopy pushout.
Precisely, we must check that for any pair $\cY$, $\cZ$ of big families in a bornological coarse space $X$, we have homotopy pushouts
\[
\begin{tikzcd}
	\Omega^{\infty+1} K\bigl(\Az((\cY \Cap \cZ) \otimes \R^n)\bigr)
	\ar[r]
	\ar[d]
	&
	\Omega^{\infty+1} K\bigl(\Az(\cY \otimes \R^n)\bigr)
	\ar[d]
	\\
	\Omega^{\infty+1} K\bigl(\Az(\cZ \otimes \R^n)\bigr)
	\ar[r]
	&
	\Omega^{\infty+1} K\bigl(\Az((\cY \Cup \cZ) \otimes \R^n)\bigr)
\end{tikzcd}
\]
of spaces (here $\cY \otimes \R^n$ etc.\ denotes the big family in $X \otimes \R^n$ generated by the sets $Y \times \R^n$ for $Y \in \cY$).
But this follows from Prop.~\ref{PropSupportedVersion} together with Lemma~\ref{LemmaHigherHomotopyGroups}.

Finally, to see $u$-continuity, observe that $\Az(X)$ is the colimit of the subcategories $\Az_{\langle E\rangle}(X)$ containing all objects and those morphisms that are controlled by an entourage in the coarse structure $\langle E\rangle$ generated by $E$.
The statement therefore follows because the $K$-theory functor commutes with colimits.
\end{proof}

\begin{remark}
Similar to what is done in \cite{QCAspace}, when defining nets, we may choose a commutative ring $R$ and replace the category of finite-dimensional $C^*$-algebras and injective $*$-homomorphisms by the category of $R$-algebras and injective homomorphisms. 
Matrix nets are then required to assign algebras isomorphic to $M_n(R)$ to bounded subsets of a bornological coarse space $X$ and Azumaya nets may be defined in a similar way, yielding a category $\Az_R(X)$.
The further constructions go through in a similar way, and we obtain a homology theory $Q(X, R)$ with $Q_{-1}(X, R) = K_0(\Az_R(X))$.

However the identification with quantum cellular automata is not quite as straight forward, as we used specific results on the projective unitary group in the proof of Thm.~\ref{ThmK1QCA}, compare \cite{QCAspace}.
\end{remark}

\subsection{QCA and Azumaya nets}

As coarse spaces, $\Z$ and $\R$ are equivalent, hence we have canonical homotopy equivalences 
\[
\Omega K(\Az(X \otimes \Z)) \simeq \Omega K(\Az(X \otimes \R))\simeq  K(\Az(X)).
\]
In view of Thm.~\ref{ThmK1QCA}, taking $\pi_0$, we obtain the following result.

\begin{theorem}
\label{ThmQCAAz}
For any bornological coarse space $X$, we have
\[
\QCA(X \otimes \Z) \cong K_1(\Az(X \otimes \Z)) \cong K_0(\Az(X)).
\]
In particular, for $X = \Z^{n-1}$, we obtain
\[
\QCA(\Z^n) \cong K_0(\Az(\Z^{n-1})).
\]
\end{theorem}

Thm.~\ref{ThmQCAAz} shows that the classification of quantum cellular automata in any dimension is equivalent to the classification of Azumaya nets one dimension lower.
This classification problem may not be much simpler, but at least in one case, it is easy to do:

\begin{example}
In the case that $X$ is a point, we have $\Az(X) \cong \Loc(X)$ and since any matrix algebra is classified, up to isomorphism, by its dimension and we have $M_n(\C) \otimes M_m(\C) \cong M_{nm}(\C)$, the monoid of isomorphism classes of objects is $(\N , \cdot)$,  natural numbers with multiplication.
The group completion yields therefore
\[
K_0(\Az(*)) \cong K_0(\Loc(*)) \cong \Q_{>0}
\]
and the isomorphism above is given by the GNVW index, see \cite{GrossNesmeVogtsWerner2012}.
\end{example}

\begin{remark}
The isomorphism from Thm.~\ref{ThmQCAAz} is just the Mayer--Vietoris boundary
\[
K_1(\Az(X \otimes \Z)) \longrightarrow K_0(\Az(X))
\]
 of the decomposition of $X \otimes \Z$ into the big families generated by $X \otimes \Z_{\geq 0}$ respectively $X \otimes \Z_{\leq 0}$.
An inverse to this map map be given as follows.
 Take an Azumaya net $\sA$ on $X$.
 By Corollary~\ref{CorollaryAzumayaSubnet}, we may realize $\sA$ as a subnet of a local matrix net $\sB \cong \sA \otimes \sA'$.
Let $\hat{\sB}$ be the local net on $X \otimes \Z$ obtained by placing a copy of $\sB$ on each point of $\Z$.
Then the image of $[\sA] \in K_0(\Az(X))$ under the above isomorphism is the QCA on $X \otimes \Z$ that acts as shift on the $\sA$ factors of $\hat{\sB}$ in the $\Z$ direction and leaves the $\sA'$ factors invariant.
\[
\begin{tikzpicture}
\node at (2,0) {$\sA$};
\node at (2,0.5) {$\otimes$};
\node at (2.1,1.05) {$\sA'$};
\node at (4,0) {$\sA$};
\node at (4,0.5) {$\otimes$};
\node at (4.1,1.05) {$\sA'$};
\node at (6,0) {$\sA$};
\node at (6,0.5) {$\otimes$};
\node at (6.1,1.05) {$\sA'$};
\node at (8,0) {$\sA$};
\node at (8,0.5) {$\otimes$};
\node at (8.1,1.05) {$\sA'$};
\node at (10,0) {$\sA$};
\node at (10,0.5) {$\otimes$};
\node at (10.1,1.05) {$\sA'$};
\node at (12,0.5) {$\dots$};
\node at (0.6,0.5) {$\dots$};
\draw[->] (2.5,0) -- (3.5,0);
\draw[->] (4.5,0) -- (5.5,0);
\draw[->] (6.5,0) -- (7.5,0);
\draw[->] (8.5,0) -- (9.5,0);
\end{tikzpicture}
\]
\end{remark}

One has the following general statement:

\begin{lemma}
Inclusion of categories induces an injective group homomorphism
\[
K_0(\Loc(X)) \longrightarrow K_0(\Az(X)).
\]
\end{lemma}

\begin{proof}
Since the inclusion functor is fully faithful, the monoid $\pi_0(\Loc(X))$ of isomorphism classes of objects is a submonoid of $\pi_0(\Az(X))$.

Generally, if $M$ is a commutative monoid, two elements $m_1, m_2 \in M$ yield the same element in the group completion $K_0(M)$ if there exists $n \in M$ such that $m_1 + n = m_2 + n$.
Hence if $\sA_1$ and $\sA_2$ are two local matrix nets that define the same element of $K_0(\Az(X))$, then there exists an Azumaya net $\sC$ such that $\sA_1 \otimes \sC \cong \sA_2 \otimes \sC$.
Therefore also $\sA_1 \otimes \sC \otimes \sC' \cong \sA_2 \otimes \sC \otimes \sC'$ if $\sC'$ is another Azumaya net such that $\sC \otimes \sC'$ is isomorphic to a local matrix net. 
But this implies that $\sA_1$ and $\sA_2$ already define the same element in $K_0(\Loc(X))$.
\end{proof}

It was shown in \cite[Thm.~D]{QCAspace} that one has a canonical isomorphism
\[
K_0(\Loc(X)) \cong H\cX_0(X, \Q_{>0}),
\]
where $H\cX_n(X)$ denotes the coarse homology groups valued in the abelian group $\Q_{>0}$ (see \cite[\S6.3]{BunkeEngel2020}) and $\Z^{\mathbb{P}}$ denotes the set of maps from the set $\mathbb{P}$ of all prime numbers to $\Z$.
The group $H\cX_0(X, \Z^{\mathbb{P}})$ is just a quotient of the group of functions $X \to \Q_{>0}$ of locally finitely support, and the isomorphism above is given by sending a local matrix $\sA$ net to its dimension function $q(x) = \dim(\sA_x)^{1/2}$, see Example~\ref{ExampleCoordinateNets}.

Since $Q_{-1}(X) = K_0(\Az(X))$,  we get a canonical map 
\[
H\cX_0(X, \Q_{>0}) \longrightarrow Q_{-1}(X).
\]
Since $H\cX_k(\Z^n) \cong H\cX_{k-n}(*)$ (as is the case for all coarse homology theories), we see that for $X = \Z^n$, the left group is zero and so does not provide any non-trivial elements.

\bibliography{literature.bib}

\begin{thebibliography}{1}

\bibitem{BunkeEngel2020}
U.~Bunke and A.~Engel.
\newblock {\em Homotopy Theory with Bornological Coarse Spaces}, volume 2269 of
  {\em Lecture Notes in Mathematics}.
\newblock Springer, 2020.

\bibitem{Goto}
M.~Goto.
\newblock {A Theorem on compact semi-simple groups}.
\newblock {\em Journal of the Mathematical Society of Japan}, 1(3):270 -- 272,
  1949.

\bibitem{GrossNesmeVogtsWerner2012}
D.~Gross, V.~Nesme, H.~Vogts, and R.~F. Werner.
\newblock Index theory of one‐dimensional quantum walks and cellular
  automata.
\newblock {\em Communications in Mathematical Physics}, 310(2):419--454, 2012.
\newblock Received 21 September 2010, accepted 7 September 2011.

\bibitem{QCAspace}
M.~Ji and B.~Yang.
\newblock Quantum cellular automata: The group, the space, and the spectrum,
  2026.
\newblock \url{https://arxiv.org/pdf/2602.16572}.

\bibitem{Mitchener2002CoarseHomology}
P.~D. Mitchener.
\newblock Coarse homology theories.
\newblock {\em Algebraic \& Geometric Topology}, 2:271--297, 2002.

\bibitem{Quillen1973HigherK}
D.~Quillen.
\newblock Higher algebraic {$K$}-theory {I}.
\newblock In H.~Bass, editor, {\em Algebraic K-Theory I: Higher K-Theories},
  volume 341 of {\em Lecture Notes in Mathematics}, pages 85--147. Springer,
  1973.

\bibitem{ThomasonOnline}
R.~W. Thomason.
\newblock First quadrant spectral sequences in algebraic {$K$}-theory.
\newblock \url{https://webhomes.maths.ed.ac.uk/~v1ranick/papers/thomason1.pdf}.

\bibitem{Thomason1982}
R.~W. Thomason.
\newblock First quadrant spectral sequences in algebraic {$K$}-theory via
  homotopy colimits.
\newblock {\em Communications in Algebra}, 10(15):1589--1668, 1982.

\end{thebibliography}

\end{document}